\documentclass{amsart}

\usepackage{amssymb}
\usepackage[all]{xy}
\usepackage{mathrsfs}
\usepackage{upgreek}
\usepackage{stmaryrd}
\usepackage{rotating}
\usepackage{tikz}
\usetikzlibrary{cd}

\usepackage[pagebackref,hyperindex,linktocpage=true]{hyperref}
\hypersetup{
    colorlinks=true,
    linkcolor={red!50!black},
    citecolor={blue!50!black},
    urlcolor={blue!50!black}
}

\usepackage{color}

\newcommand{\bk}{\Bbbk}
\newcommand{\Z}{\mathbb{Z}}
\newcommand{\R}{\mathbb{R}}

\newcommand{\Q}{\mathbb{Q}}
\newcommand{\F}{\mathbb{F}}
\newcommand{\scO}{\mathscr{O}}

\DeclareMathOperator{\Hom}{Hom}

\newcommand{\id}{\mathrm{id}}
\newcommand{\simto}{\xrightarrow{\sim}}

\newcommand{\fR}{\mathfrak{R}}

\newcommand{\bG}{\mathbf{G}}
\newcommand{\bH}{\mathbf{H}}
\newcommand{\bB}{\mathbf{B}}
\newcommand{\bT}{\mathbf{T}}
\newcommand{\bC}{\mathbf{C}}
\newcommand{\bU}{\mathbf{U}}
\newcommand{\bN}{\mathbf{N}}
\newcommand{\bW}{\mathbf{W}}
\newcommand{\bP}{\mathbf{P}}
\newcommand{\bM}{\mathbf{M}}
\newcommand{\bZ}{\mathbf{Z}}

\makeatletter
\def\lotimes{\@ifnextchar_{\@lotimessub}{\@lotimesnosub}}
\def\@lotimessub_#1{\mathchoice{\mathbin{\mathop{\otimes}^L}_{#1}}%
  {\otimes^L_{#1}}{\otimes^L_{#1}}{\otimes^L_{#1}}}
\def\@lotimesnosub{\mathbin{\mathop{\otimes}^L}}
\makeatother

\makeatletter
\def\lboxtimes{\@ifnextchar_{\@lboxtimessub}{\@lboxtimesnosub}}
\def\@lboxtimessub_#1{\mathchoice{\mathbin{\mathop{\boxtimes}^L}_{#1}}%
  {\boxtimes^L_{#1}}{\boxtimes^L_{#1}}{\boxtimes^L_{#1}}}
\def\@lboxtimesnosub{\mathbin{\mathop{\boxtimes}^L}}
\makeatother

\newcommand{\Spec}{\mathrm{Spec}}
\newcommand{\Diag}{\mathrm{D}}
\mathchardef\mhyphen="2D

\newcommand{\Lie}{\mathscr{L}\hspace{-1.5pt}\mathit{ie}}
\newcommand{\Roots}{\mathfrak{R}}

\newcommand{\sconn}{\mathrm{sc}}

\numberwithin{equation}{section}
\newtheorem{thm}{Theorem}[section]
\newtheorem{lem}[thm]{Lemma}
\newtheorem{prop}[thm]{Proposition}
\newtheorem{cor}[thm]{Corollary}

\theoremstyle{definition}

\theoremstyle{remark}
\newtheorem{rmk}[thm]{Remark}
\newtheorem{ex}[thm]{Example}

\title[Pinned fixed points of reductive groups]{Fixed points under pinning-preserving automorphisms of reductive group schemes}

\author{Pramod N. Achar}
\address{Department of Mathematics\\
  Louisiana State University\\
  Baton Rouge, LA 70803\\
  U.S.A}
\email{pramod@math.lsu.edu}

\author{Jo{\~a}o Louren{\c c}o}
	\address{Mathematisches Institut, Universität Münster, Einsteinstrasse 62, Münster, Germany}
\email{j.lourenco@uni-muenster.de}

\author{Timo Richarz}
\address{Technische Universit\"at Darmstadt, Department of Mathematics, 64289 Darmstadt, Germany}
\email{richarz@mathematik.tu-darmstadt.de}

\author{Simon Riche}
\address{Universit\'e Clermont Auvergne, CNRS, LMBP, F-63000 Clermont-Ferrand, France}
\email{simon.riche@uca.fr}

\thanks{P.A. was supported by NSF Grant Nos.~DMS-1802241 and DMS-2202012. J.L. was supported by the Max-Planck-Institut für Mathematik, the Excellence Cluster of the Universität Münster, and by the ERC Consolidator Grant 770936 via Eva Viehmann. This project has received
funding from the European Research Council (ERC) under the European Union's Horizon 2020
research and innovation programme (S.R. \& T.R., grant agreement No 101002592),
funding (T.R.) from the Deutsche Forschungsgemeinschaft (DFG, German Research Foundation) TRR 326 \textit{Geometry and Arithmetic of Uniformized Structures}, project number 444845124 and funding (T.R.) from the LOEWE professorship in algebra. 
}

\begin{document}

\begin{abstract}
In this paper we determine the scheme-theoretic fixed points of pinned reductive group schemes acted upon by a group of pinning-preserving automorphisms.
The results are used in a companion paper to establish a ramified geometric Satake equivalence with integral or modular coefficients. 
\end{abstract}

\maketitle

\section{Introduction}

\subsection{}

In this paper we study some basic properties (e.g.~flatness and smoothness) of the fixed point group scheme for an action of an abstract group on a split reductive group scheme by pinning preserving automorphisms. 
Our motivation for such a study comes from work on a ramified version of the geometric Satake equivalence; in fact, in the companion paper~\cite{alrr} we show that such fixed points group schemes (over the spectrum of the $\ell$-adic integers $\Z_\ell$ or the finite field $\F_\ell$ for some prime number $\ell$) arise as Tannakian groups for appropriate categories of equivariant perverse sheaves on affine Grassmannians attached to  parahoric groups associated with special facets of Bruhat--Tits buildings. 
(An analogous study over $\Q_\ell$ was undertaken by Zhu~\cite{zhu} and by the third named author~\cite{richarz}.)
We believe that our study is of independent interest. 
It is treated here in greater generality and in more detail than what is actually needed in~\cite{alrr}.

\subsection{Statement}

Let $S$ be a nonempty scheme, and consider a pinned reductive group $(\bG,\bT,M,\Roots,\Delta,X)$ over $S$. In particular, $\bT=\Diag_S(M)$ is a (split) maximal torus in $\bG$, $\Roots$ is the associated root system, $\Delta$ is a basis of $\Roots$, and $X$ is a collection of nonvanishing sections of the root subspaces in $\Lie(\bG)$ attached to simple roots. (In case $S$ is the spectrum of a field or a principal ideal domain, one can take for $\bG$ any reductive group scheme admitting a split maximal torus $\bT$; then $M$ is the lattice of characters of $\bT$, and a pinning $X$ always exists because the weight spaces in $\Lie(\bG)$ are free over $\scO(S)$.) We then have a corresponding group $\mathrm{Aut}(\bG,\bT,M,\Roots,\Delta,X)$ of automorphisms of $\bG$ preserving these data (see~\S\ref{ss:pinned-red-groups-aut} for precise references), which identifies with the group of automorphisms of the associated root datum. We consider an abstract group $A$ and an action of $A$ on $\bG$ defined by a homomorphism $A \to \mathrm{Aut}(\bG,\bT,M,\Roots,\Delta,X)$.

Our main result is the following (see Theorem~\ref{thm:fixed-pts} and Proposition~\ref{prop:fixed-pts-fields-2}).  In this statement, $M_A$ denotes the group of coinvariants for the action of $A$ on $M$ (see~\S\ref{ss:ex-diag-gps}).

\begin{thm}
\phantomsection
\label{thm:main-intro}
 \begin{enumerate}
  \item 
  \label{it:main-intro-1}
  The group scheme $\bG^A$ is flat over $S$.
  \item 
  \label{it:main-intro-2}
  The group scheme $\bG^A$has geometrically connected fibers over $S$ iff either $M_A$ is torsion free or 
  $S$ has exactly one residual characteristic $\ell>0$ and the torsion part of $M_A$ is an $\ell$-group.
  \item 
  \label{it:main-intro-3}
  The group scheme $\bG^A$ is smooth over $S$ iff the following conditions hold:
  \begin{enumerate}
  \item the order of the torsion subgroup of $M_A$ is prime to all residual characteristics of $S$;
  \item 
  \label{it:smoothness-issue-intro}
  if $\Roots$ has an indecomposable component of type $\mathsf{A}_{2n}$ for some $n \geq 1$ whose stabilizer in $A$ acts nontrivially on this component, then $2$ is not a residual characteristic of $S$.
  \end{enumerate}
  \item
  \label{it:main-intro-4}
  If $S=\Spec(\bk)$ for some field $\bk$, then the reduced neutral component $(\bG^A)^\circ_{\mathrm{red}}$ is a split reductive group scheme, and if $\overline{\bk}$ is an algebraic closure of $\bk$ one has
  \[
  ((\overline{\bk} \otimes_\bk \bG)^A)^\circ_{\mathrm{red}}=\overline{\bk} \otimes_{\bk} (\bG^A)^\circ_{\mathrm{red}}.
  \]
  \end{enumerate}
\end{thm}

Let us first discuss some previous occurrences of such results in the literature.
Statement~\eqref{it:main-intro-4} was already known. The case when $\bk$ is algebraically closed and $A$ is finite and cyclic (and satisfies a condition weaker than preserving a pinning) was treated by Steinberg~\cite{steinberg}; see also~\cite[Th\'eor\`eme~1.8]{digne-michel} for a further study. A more general version (not requiring $\bG$ to be split, and also replacing the existence of a fixed pinning by a weaker condition) is due to Adler--Lansky, see~\cite[Proposition~3.5]{adler-lansky}. In case $\bk$ is algebraically closed and $A$ is finite, the same statement as ours appears in work of Haines~\cite{haines}. All of these proofs are based on~\cite{steinberg}. We give a new proof of this statement here, which does not rely on~\cite{steinberg} except for elementary claims on root systems.

The smoothness of $\bG^A$ (in case $S=\Spec(\Z[\frac{1}{2}])$, $\bG$ is semisimple and simply-connected and $A$ is cyclic) is also proved in~\cite[Lemma~4.25]{dhkm}; the nonsmoothness in the setting of condition~\eqref{it:smoothness-issue-intro} is mentioned in~\cite[Remark~4.26]{dhkm}. In case $S=\Spec(\bk)$ for a field $\bk$, and $A$ is finite of cardinality invertible in $\bk$ (but does not necessarily fix a pinning), the fact that $(\bG^A)^\circ$ is reductive (in particular, smooth) is also proved in~\cite[Theorem~2.1]{prasad-yu-finite-gp}.

Let us point out also that if $S$ is the spectrum of a mixed characteristic discrete valuation ring and the coinvariants $M_A$ of the action of $A$ on $M$ are torsion-free, then $\bG^{A}$ is a quasi-reductive $S$-group scheme in the sense of Prasad--Yu~\cite{py} (see Theorem~\ref{thm:fixed-pts}\eqref{it:fixed-pts-5}).
A particularly interesting example is when $S=\Spec(\Z_2)$ and $\bG=\mathrm{SL}_{2n+1,\Z_2}$ for some $n \geq 1$, with the unique nontrivial action of $\Z/2\Z$. 
In this case, $\bG^A\to S$ is nonsmooth by~\eqref{it:main-intro-3}, and hence in particular nonreductive.

\subsection{Outline of the proof}

The main step in the proof of Theorem~\ref{thm:main-intro} consists of an analysis of the fixed points of $A$ on the big cell in $\bG$ attached to our given pinning. For this we study separately the fixed points on $\bT$ (which is rather straightforward) and on the (positive and negative) ``maximal unipotent subgroups'' $\bU$ and $\bU^-$. This part is more subtle, and requires the construction of an appropriate ``extension'' of $X$ to a Chevalley system compatible (in the appropriate sense) with the action of $A$. We also analyze the case when $S$ is the spectrum of an algebraically closed field in great detail in~\S\ref{ss:fixed-pts-alg-closed-fields}.

Let us note that the groups considered in Theorem~\ref{thm:main-intro} share many standard properties of reductive group schemes, although they are not reductive in general. In particular: 
\begin{itemize}
\item
to such a group we attach a root system (and even several root data, see~\S\ref{ss:root-data});
\item
in~\S\ref{ss:SL2-maps} we construct certain ``twisted $\mathrm{SL}_2$-maps'' associated with positive roots in this root system (whose domain is not necessarily $\mathrm{SL}_{2,S}$);
\item
in~\S\ref{ss:Weyl-gp} we show that the quotient of the normalizer of $\bT^A$ by its centralizer is (the constant group scheme attached to) a Coxeter group, which identifies with the Weyl group of the associated root system;
\item
in~\S\ref{ss:parabolic-Levi} we study analogues of parabolic and Levi subgroups in $\bG^A$.
\end{itemize}
Here again, in the special case when $S$ is the spectrum of a field, some of these constructions appear in work of Adler--Lansky and Haines, sometimes under weaker assumptions; see in particular~\cite{adler-lansky-2,haines2} for discussions of root data attached to fixed points.

\subsection{Contents}

In Section~\ref{sec:fixed-pts} we recall the definition of fixed-point schemes, and study some first examples. In Section~\ref{sec:pinned-red-gp-sch} we recall the definition of pinned reductive group schemes, and some basic results on their structure. In Section~\ref{sec:equiv-roots} we construct our twisted $\mathrm{SL}_2$-maps. In Section~\ref{sec:flatness-smoothness} we prove Theorem~\ref{thm:main-intro}. Finally, in Section~\ref{sec:complements} we prove some complements, some of which will be used in~\cite{alrr}.

\subsection{Acknowledgements}

We thank J.~Adler, B.~Conrad, S.~Cotner, J.~F.~Dat and P.~Gille for interesting discussions on the subject of this paper.

\section{Fixed points}
\label{sec:fixed-pts}

\subsection{Definition}
\label{ss:def-fixed-pts}

Let $S$ be a scheme, and let $A$ be an abstract group. Recall that if $X \to S$ is an $S$-scheme, endowed with an action of $A$ by $S$-scheme automorphisms, then the functor of $A$-fixed points $X^A$ is defined as the functor that sends an $S$-scheme $Y$ to the set $\Hom_S(Y,X)^A$, where $A$ acts on the set of $S$-scheme morphisms $\Hom_S(Y,X)$ via its action on $X$, and the superscript means fixed points in the usual sense. (In other words, $X^A$ is the fixed-points sheaf associated with the natural action of $A$ on the fpqc sheaf on the category of $S$-schemes represented by $X$.) If this functor is representable by a scheme, then this scheme will also be denoted $X^A$. It is clear from the definition that this construction is stable under base change; namely, if $S' \to S$ is a morphism of schemes then we have an identification of functors
\begin{equation}
\label{eqn:fixed-pts-bc}
 S' \times_S X^A \simto (S' \times_S X)^A
\end{equation}
where on the right-hand side we consider the natural $A$-action on $S' \times_S X$ induced by the action on $X$. In particular, if $X^A$ is representable by a scheme, then so is $(S' \times_S X)^A$. It is also clear that given an open covering $S = \bigcup_{i \in I} S_i$, the functor $X^A$ is representable by a scheme if and only if for any $i$ the functor $(S_i \times_S X)^A$ is representable by a scheme. (In this case we have an open covering $X^A = \bigcup_{i \in I} (S_i \times_S X)^A$.) This construction is functorial in the sense that if $f : X \to Y$ is an $A$-equivariant morphism of schemes we have a canonical morphism of functors $f^A : X^A \to Y^A$.

The following properties are easily verified from the definitions.

\begin{lem}
\phantomsection
\label{lem:fixed-pts-properties}
\begin{enumerate}
\item
\label{it:fixed-pts-immersion}
Let $X$ and $Y$ be $S$-schemes endowed with actions of $A$, let $f : X \to Y$ be an $A$-equivariant monomorphism over $S$, and let $f^A : X^A \to Y^A$ be the induced morphism. Then the following diagram is cartesian:
\[
\xymatrix{
X^A \ar[r]^-{f^A} \ar[d] & Y^A \ar[d] \\ X \ar[r]^-{f} & Y.
}
\]
In particular, if $Y^A$ is representable by a scheme then so is $X^A$, and if $f$ is an open immersion, resp.~a closed immersion, resp.~an immersion, then so is $f^A$.
\item
\label{it:fixed-pts-product}
Let $X$ and $Y$ be $S$-schemes endowed with actions of $A$, and consider the diagonal action on $X \times_S Y$. Then we have a canonical identification
\[
(X \times_S Y)^A = X^A \times_S Y^A.
\]
\end{enumerate}
\end{lem}

For general results on fixed-point schemes the reader is referred to \cite{Fogarty_FixedPointSchemes}.
Below we will only consider the case when the morphism $X \to S$ is affine. 
Recall that under this assumption $X^A$ is always representable by a scheme, and the natural morphism $X^A \to X$ is a closed immersion. 
Indeed, passing to an open covering we can assume that $S$ is affine, say $S = \Spec(k)$ for some ring $k$. 
Then $X$ is also affine, say $X = \Spec(R)$, where $R$ is a $k$-algebra.  
The action of $A$ on $X$ corresponds to an action on $R$ by $k$-algebra automorphisms.  
It is easily checked that $X^A$ is represented by the closed subscheme $\Spec(R/I) \subset X$, where $I \subset R$ is the ideal generated by elements of the form $r - a \cdot r$ for $r \in R$ and $a \in A$.

In case $X$ is an affine group scheme over $S$ and $A$ acts by group scheme automorphisms, then of course $X^A$ is a (closed) subgroup scheme of $X$.

\subsection{The case of diagonalizable groups}
\label{ss:ex-diag-gps}

Let us now study the construction of~\S\ref{ss:def-fixed-pts} for certain actions on diagonalizable group schemes.
Recall that for any scheme $S$ and any abelian group $M$, we have an associated group scheme
\[
\Diag_S(M)
\]
representing the group valued functor $\underline{\textrm{Hom}}_{S\textrm{-GrpSch}}(\underline{M}_S,\bG_{\mathrm{m},S})$ on the category of $S$-schemes, see e.g.~\cite[Exp.~VIII, D\'efinition~1.1]{sga3.2}. 
Consider a group $A$ and an action of $A$ on $M$ by group automorphisms. 
We deduce an action of $A$ on the group scheme $\Diag_S(M)$ over $S$, by group scheme automorphisms. 
(By~\cite[Exp.~VIII, Corollaire~1.6]{sga3.2}, any action of $A$ on $\Diag_S(M)$ by group scheme automorphisms arises in this way in case $M$ is of finite type and $S$ is connected.) 
We will denote by $M_A$ the group of coinvariants for this action, i.e.~the quotient of $M$ by the subgroup generated by the elements of the form $m - a \cdot m$ for $a \in A$ and $m \in M$.

\begin{lem}
\label{lem:torus_fix_pts_diag}
There exists a canonical isomorphism of group schemes over $S$
\[
\Diag_S(M_A) \simto \bigl( \Diag_S(M) \bigr)^A.
\]
\end{lem}

\begin{proof}
The projection $M \to M_A$ is $A$-equivariant for the trivial action on the right-hand side; it therefore induces an $A$-equivariant morphism $\Diag_S(M_A) \to \Diag_S(M)$, which necessarily factors through a morphism $\Diag_S(M_A) \to ( \Diag_S(M) )^A$. To prove that this morphism is an isomorphism we can assume that $S=\Spec(k)$ is affine. Then $( \Diag_S(M) )^A$ is the spectrum of the quotient $R$ of the group algebra $k[M]$ by the ideal generated by the elements of the form $x - a \cdot x$ for $a \in A$ and $x \in k[M]$. For any $k$-algebra $R'$ we have
\begin{multline*}
\Hom_{k\mhyphen\mathrm{alg}}(R,R') = \Hom_{k\mhyphen\mathrm{alg}}(k[M],R')^A = \Hom_{\mathrm{gps}}(M,(R')^\times)^A \\
= \Hom_{\mathrm{gps}}(M_A,(R')^\times) = \Hom_{k\mhyphen\mathrm{alg}}(k[M_A],R')
\end{multline*}
where $k\mhyphen\mathrm{alg}$ is the category of $k$-algebras, $\mathrm{gps}$ is the category of groups, and in all cases the $A$-action is induced in the natural way by the action on $M$.
We deduce an identification $R=k[M_A]$, which finishes the proof.
\end{proof}

Lemma~\ref{lem:torus_fix_pts_diag} shows that $( \Diag_S(M) )^A$ is always flat over $S$. If we assume that $M$ is of finite type, then by~\cite[Exp.~VIII, Proposition~2.1(e)]{sga3.2} this group scheme is smooth over $S$ iff the order of the torsion subgroup of $M_A$ is prime to all residual characteristics of $S$. Still assuming that $M$ is of finite type, $( \Diag_S(M) )^A$ is geometrically connected iff either $M_A$ is torsion free or 
$S$ has exactly one residual characteristic $\ell>0$ and the torsion part of $M_A$ is an $\ell$-group.

\subsection{The case of \texorpdfstring{$\mathrm{SL}_{2n+1}$}{SL2n+1}}

Let $n \ge 1$, and consider the group scheme $\mathrm{SL}_{2n+1,\Z}$ over $\Spec(\Z)$.  
In this subsection, we study the construction of~\S\ref{ss:def-fixed-pts} for a certain action of the group $A = \Z/2\Z$ on $\mathrm{SL}_{2n+1,\Z}$.

\subsubsection{Action}
\label{sss:action-SLodd}

Denote by $J_{2n+1}$ the square matrix of size $2n+1$ whose coefficient in position $(i,j)$ is given by:
\begin{itemize}
\item $0$ if $i+j \neq 2n+2$;
\item $1$ if $i+j=2n+2$ and $i$ is even;
\item $-1$ if $i+j=2n+2$ and $i$ is odd.
\end{itemize}
That is, $J_{2n+1}$ has entries $(-1,1,-1,\ldots,-1, 1,-1)$ on the anti-diagonal and $0$ else, so $J_{2n+1}^2=\id$.
We then make $A=\Z/2\Z$ act on $\mathrm{SL}_{2n+1,\Z}$ by having the nontrivial element act by
\[
M \mapsto J_{2n+1} \cdot {}^{\mathrm{t}} \hspace{-1pt} M^{-1} \cdot J_{2n+1}.
\]

\subsubsection{The case of \texorpdfstring{$\mathrm{SL}_3$}{SL3}}
\label{sss:sl3_action}

First, let us consider the case $n=1$, and denote by~$\mathrm{U}_{3,\Z}$ the subgroup scheme of $\mathrm{SL}_{3,\Z}$ consisting of upper triangular unipotent matrices.
For any ring $R$, the action of the nontrivial element of $A$ on $\mathrm{U}_{3,\Z}(R)$ is given by
\[
 \begin{pmatrix}
  1 & x & y \\ 0 & 1 & z \\ 0 & 0 & 1 
 \end{pmatrix}
\mapsto
\begin{pmatrix}
  1 & z & xz-y \\ 0 & 1 & x \\ 0 & 0 & 1 
 \end{pmatrix}.
\]
It follows that the fixed-point subscheme $(\mathrm{U}_{3,\Z})^{\Z/2\Z}$ is the closed subgroup scheme defined by the equations
\[
 x=z, \quad xz-y=y;
\]
we therefore have
\[
\label{eqn:fixed-pts-SL3}
 (\mathrm{U}_{3,\Z})^{\Z/2\Z} \cong \Spec \bigl( \Z[x,y] / (x^2-2y) \bigr).
\]
Since $x^2 - 2y$ is monic as a polynomial in $x$, we see that $(\mathrm{U}_{3,\Z})^{\Z/2\Z}$ is finite and flat over $\Spec(\Z[y]) = \mathbb{A}^1_{\Z}$, and hence flat over $\Z$.  More generally, for an arbitrary nonempty scheme $S$, considering the group scheme
\[
\mathrm{U}_{3,S} := S \times_{\Spec(\Z)} \mathrm{U}_{3,\Z},
\]
in view of~\eqref{eqn:fixed-pts-bc}
we see that
\[
(\mathrm{U}_{3,S})^{\Z/2\Z} = S \times_{\Spec(\Z)} (\mathrm{U}_{3,\Z})^{\Z/2\Z}
\]
is flat over $S$.

We claim that $(U_{3,S})^{\Z/2\Z}$ is smooth over $S$ if and only if $2$ is not a residual characteristic of $S$. Indeed, if $2$ is a residual characteristic, and if $s \in S$ is such that the residue field $\kappa(s)$ has characteristic $2$, then
\[
\Spec(\kappa(s)) \times_S (\mathrm{U}_{3,S})^{\Z/2\Z} = \Spec(\kappa(s)[x,y]/x^2)
\]
is not reduced, and hence not smooth (see~\cite[\href{https://stacks.math.columbia.edu/tag/056T}{Tag 056T}]{stacks-project}), so $(\mathrm{U}_{3,S})^{\Z/2\Z}\to S$ is not smooth. To prove the converse implication we can assume
that $S$ is affine, say $S = \Spec(k)$. If $2$ is not a residual characteristic of $S$, then it is invertible in $k$, so $k \to k[x,y]/(x^2 - 2y)$ is a standard smooth ring map in the sense of~\cite[\href{https://stacks.math.columbia.edu/tag/00T6}{Tag 00T6}]{stacks-project} (because $\frac{\partial}{\partial y} (x^2 - 2y)$ is invertible), and finally $(\mathrm{U}_{3,S})^{\Z/2\Z} \to S$ is smooth by~\cite[\href{https://stacks.math.columbia.edu/tag/00T7}{Tag 00T7}]{stacks-project}.  

If $S = \Spec(\bk)$ is the spectrum of a field, the reduced subscheme $(\mathrm{U}_{3,\bk})^{\Z/2\Z}_{\mathrm{red}}$ is given by
\begin{equation}\label{eqn:u3-fixed-field}
(\mathrm{U}_{3,\bk})^{\Z/2\Z}_{\mathrm{red}}
=
\left\{\left(\begin{smallmatrix} 1 & x & x^2/2 \\ 0 & 1 & x \\ 0 & 0 & 1\end{smallmatrix}\right)\right\}\text{ if $\mathrm{char}(\bk) \ne 2$,} \qquad
\left\{\left(\begin{smallmatrix} 1 & 0 & y \\ 0 & 1 & 0 \\ 0 & 0 & 1\end{smallmatrix}\right)\right\}\text{ if $\mathrm{char}(\bk) = 2$.}
\end{equation}

Over $\Z[\frac{1}{2}]$, the map of group schemes $\textrm{SL}_{2,\mathbb Z[\frac{1}{2}]}\to (\mathrm{SL}_{3,\Z[\frac{1}{2}]})^{\Z/2\Z}$ given explicitily by 
\[
\begin{pmatrix}
a & b \\ c & d
\end{pmatrix} \mapsto
\begin{pmatrix}
a^2 & ab & \frac{1}{2} b^2 \\
2ac & ad+bc & bd \\
2c^2 & 2cd & d^2
\end{pmatrix}
\]
induces a closed immersion of group schemes
\begin{equation}
\label{eqn:SL3-fixed-pts}
\xi: \mathrm{PGL}_{2,\Z[\frac{1}{2}]} \to (\mathrm{SL}_{3,\Z[\frac{1}{2}]})^{\Z/2\Z}.
\end{equation}
(This morphism is induced by the adjoint action of $\mathrm{SL}_{2,\Z[\frac{1}{2}]}$ on its Lie algebra using the ordered basis 
$\left(\begin{smallmatrix}0 & -2 \\ 0 & 0\end{smallmatrix}\right), 
\left(\begin{smallmatrix}1 & 0 \\ 0 & -1\end{smallmatrix}\right), 
\left(\begin{smallmatrix}0 & 0 \\ 1 & 0\end{smallmatrix}\right)$.)
On the other hand, we will see below that $(\mathrm{SL}_{3,\F_2})^{\Z/2\Z}$ is nonreduced; we have a closed immersion of group schemes
\begin{equation}
\label{eqn:SL3-fixed-pts-2}
\xi: \mathrm{SL}_{2,\F_2} \to (\mathrm{SL}_{3,\F_2})^{\Z/2\Z}_{\mathrm{red}}
\end{equation}
where the right-hand side is the reduced group scheme associated with $(\mathrm{SL}_{3,\F_2})^{\Z/2\Z}$,
given explicitly by
\[
\begin{pmatrix}
a & b \\ c & d
\end{pmatrix} \mapsto
\begin{pmatrix}
a & 0 & b \\
0 & 1 & 0 \\
c & 0 & d
\end{pmatrix}.
\]
In fact, the maps \eqref{eqn:SL3-fixed-pts} and \eqref{eqn:SL3-fixed-pts-2} are isomorphisms as we show in Example \ref{ex:sl_odd}\eqref{ex:sl_odd.1}.

\subsubsection{Some morphisms}
Now, let us return to the case of a general $n \geq 1$.  
For any ring $R$, let $e_1, e_2, \ldots, e_{2n+1}$ be the standard basis of the free module $R^{2n+1} = \mathbb{A}_{\Z}^{2n+1}(R)$.  
For any $i \in \{1, \dots, n\}$, define an embedding
\[
\varphi_i: \mathbb{A}_{\Z}^3 \to \mathbb{A}_{\Z}^{2n+1} \qquad\text{by}\qquad
\begin{cases}
e_1 \mapsto e_i, \\
e_2 \mapsto e_{n+1}, \\
e_3 \mapsto (-1)^{i+n}e_{2n+2-i}
\end{cases}
\]
Let $M_i = \mathrm{span}_R\ \{e_j : j \notin \{i, n+1, 2n+2-i\}\}$.  Then $R^{2n+1} = \mathrm{image}(\varphi_i(R)) \oplus M_i$. 

Make $\mathrm{SL}_{3,\Z}(R)$ act on $R^{2n+1}$ by having it act trivially on $M_i$, and by the natural action on $\mathbb{A}^3_{\Z}(R)$ (transported across $\varphi_i$).  
This action defines a closed immersion of group schemes
\begin{equation}\label{eqn:tilde-morph-SL3}
f_{i,n} : \mathrm{SL}_{3,\Z} \to \mathrm{SL}_{2n+1,\Z}.
\end{equation}
Explicitly, $f_{i.n}$ sends a matrix $M = (m_{rs})_{1 \le r,s\le 3}$ in $\mathrm{SL}_{3,\Z}(R)$ to the $(2n+1) \times (2n+1)$ matrix whose entry in position $(j,k)$ is given by the following table:
\[
\hbox{\small$\begin{array}{llcll}
\textit{Position} & \textit{Entry} && \textit{Position} & \textit{Entry}\\
\cline{1-2}\cline{4-5}
(i,i) & m_{11} &&
  (2n+2-i,i) & (-1)^{i+n}m_{31} \\
(i,n+1) & m_{12} &&
  (2n+2-i,n+1) & (-1)^{i+n}m_{32} \\
(i,2n+2-i) & (-1)^{i+n} m_{13} &&
  (2n+2-i,2n+2-i) & m_{33} \\
(n+1,i) & m_{21} &&
  (j,j),\ j \notin \{i,n+1,2n+2-i\} & 1 \\
(n+1,n+1) & m_{22} &&
  \text{all other entries} & 0 \\
(n+1, 2n+2-i) & (-1)^{i+n} m_{23}
\end{array}$}
\]
One checks by explicit computation that this morphism is equivariant with respect to the actions of $\Z/2\Z$ considered above; it therefore restricts to a morphism of group schemes
\[
(\mathrm{SL}_{3,\Z})^{\Z/2\Z} \to (\mathrm{SL}_{2n+1,\Z})^{\Z/2\Z}.
\]

\section{Pinned reductive group schemes}
\label{sec:pinned-red-gp-sch}

\subsection{Definition}
\label{ss:pinned-red-groups}

Let $S$ be a nonempty scheme, and let $(\bG,\bT, M, \Roots, \Delta, X)$ be a pinned reductive group scheme over $S$ in the sense of~\cite[Exp.~XXIII, Definition~1.1]{sga33}. In concrete terms:
\begin{itemize}
\item
 $\bT$ is a maximal torus of $\bG$, $M$ is a free abelian group of finite rank, and we are given an isomorphism of $S$-group schemes $\bT \cong \Diag_S(M)$;
 \item
 $\Roots \subset M$ is a root system of $\bG$ with respect to $\bT$ such that $(M,\Roots)$ defines a splitting of $\bG$ in the sense of~\cite[Exp.~XXII, D\'efinition~1.13]{sga33};
 \item
 $\Delta \subset \Roots$ is a system of simple roots;
 \item
  $X = (X_\alpha : \alpha \in \Delta)$ is a collection of elements in $\Lie(\bG)$ such that each $X_\alpha$ is a nowhere vanishing section of the invertible $\scO_S$-module $\Lie(\bG)^\alpha$.
\end{itemize}
The datum of $\Delta \subset \Roots$ determines a subset of positive roots, which will be denoted $\Roots_+$.

Let us comment briefly on these data, following~\cite[Exp.~XXII, Proposition~2.2]{sga33}.
Consider a reductive group scheme $\bG$ over $S$ with a maximal torus $\bT$. Saying that $\bT$ is split is the same as saying that there exists a free abelian group $M$ and an isomorphism $\bT \cong \Diag_S(M)$. 
If $S$ is connected, then $M$ is canonically determined by $\bT$, since it identifies with the group of $S$-scheme morphisms from $\bT$ to $\mathbb{G}_{\mathrm{m},S}$. 

Next, 
we have the functor of roots $\mathcal{R}$ of $\bG$ with respect to $\bT$, which is a locally constant, finite scheme, realized as an open and closed subscheme of the group scheme $\underline{\textrm{Hom}}_{S\textrm{-GrpSch}}(\bT,\bG_{\mathrm{m},S})$, see~\cite[Exp.~XIX, Proposition~3.8]{sga33}. 
Given a subset $\Roots$ of the group of morphisms of $S$-group schemes from $\bT$ to $\mathbb{G}_{\mathrm{m},S}$,
 $\Roots$ is a root system for $\bG$ with respect to $\bT$ iff the canonical inclusions of $\Roots_S$ and $\mathcal{R}$ into $\underline{\textrm{Hom}}_{S\textrm{-GrpSch}}(\bT,\bG_{\mathrm{m},S})$ induce an isomorphism $\Roots_S \to \mathcal{R}$.

We now observe
that when $\bG$ admits a split maximal torus $\bT \cong \Diag_S(M)$ and $S$ is connected, a root system always exists, and is unique.
Indeed, the action of $\bT$ on $\Lie(\bG)$ determines an $M$-grading on this $\scO_S$-module. If $\Roots \subset M$ is the set of nonzero weights for this action, for any $\alpha \in \Roots$ the corresponding $\alpha$-weight space $\Lie(\bG)^\alpha$ is a direct summand in $\Lie(\bG)$, and hence a locally free sheaf. Since its rank is locally constant (see~\cite[\href{https://stacks.math.columbia.edu/tag/01C9}{Tag 01C9}]{stacks-project}), it has to be constant (and positive), so $\alpha$ is a root for $\bG$ by~\cite[Exp.~XIX, D\'efinition~3.2]{sga33}. In view of~\cite[Exp.~XIX, D\'efinition~3.2]{sga33}, $\Roots$ is therefore a root system for $\bG$ with respect to $\bT$. 

Let us continue with the assumptions of the previous paragraph. As explained above, for any $\alpha \in \Roots$ the root subspace $\Lie(\bG)^\alpha$ is a locally free $\scO_S$-module of rank one, see~\cite[Exp.~XIX, \S 3.4]{sga33}. 
The pair $(M,\Roots)$ defines a splitting of $\bG$ with respect to $\bT$ in the sense of~\cite[Exp.~XXII, D\'efinition~1.13]{sga33} iff each $\Lie(\bG)^\alpha$ is free, which is automatic e.g.~if $\mathrm{Pic}(S)$ is trivial. Under this assumption, of course a collection $(X_\alpha : \alpha \in \Delta)$ of nowhere vanishing sections of the root subgroups attached to any choice of simple roots exists.

In conclusion, in case $S$ is connected and $\mathrm{Pic}(S)$ is trivial (e.g.~if $S$ is the spectrum of a principal ideal domain), the datum of a pinned reductive group scheme over $S$ is equivalent to the datum of a reductive group with a given split maximal torus, a system of simple roots, and a collection of nowhere vanishing sections of the associated simple root subspaces.

\begin{ex}
\label{ex:roots-A2n}
Below we will use the standard pinning of the group scheme $\bG=\mathrm{SL}_{2n+1,\Z}$ over $\Spec(\Z)$ (for $n \geq 1$). In this case: 
\begin{itemize}
\item
$\bT$ is the subgroup of diagonal matrices;
\item
$M$ is the quotient $\Z^{2n+1} / \Delta \Z$ (where $\Delta \Z$ is the diagonal copy of $\Z$ in $\Z^{2n+1}$);
\item
$\Roots = \{[\varepsilon_i - \varepsilon_j] : i \neq j \in \{1, \dots, 2n+1\} \}$ (where $(\varepsilon_1, \dots, \varepsilon_{2n+1})$ is the standard basis of $\Z^{2n+1}$, and $[\lambda]$ is the class of an element $\lambda \in \Z^{2n+1}$ in $M$);
\item
$\Delta = \{[\varepsilon_i - \varepsilon_{i+1}] : i \in \{1, \dots, 2n\}\}$;
\item
if $\alpha = [\varepsilon_i - \varepsilon_{i+1}]$, then $X_\alpha$ is the matrix whose unique nonzero coefficient is $1$ in position $(i,i+1)$.
\end{itemize}
\end{ex}

\subsection{Automorphisms}
\label{ss:pinned-red-groups-aut}

Let us now come back to the case of a general base scheme $S$. From now on in this section we fix
a pinned reductive group scheme
\begin{equation}
\label{eqn:pinned-gp}
(\bG,\bT, M, \Roots, \Delta, X)
\end{equation}
over $S$. We can then consider the group
\[
\mathrm{Aut}(\bG,\bT, M, \Roots, \Delta, X)
\]
of automorphisms $f: \bG \to \bG$ of $\bG$ that preserve the given pinning, in the sense of~\cite[Exp.~XXIII, D\'efinition~1.3]{sga33}. By definition (see in particular~\cite[Exp.~XXII, D\'efinition~4.2.1]{sga33}), any such automorphism $f$ restricts to an automorphism of $\bT$ induced by an automorphism of $M$,\footnote{As noted in~\S\ref{ss:ex-diag-gps}, in case $S$ is connected, the condition that the restriction to $\bT$ is induced by an automorphism of $M$ is automatically satisfied.} which preserves $\Roots$ and $\Delta$, and it permutes the collection $X$ according to its action on $\Delta$. In fact, as noted in~\cite[Exp.~XXII, Remarque~4.2.2]{sga33}, $f$ determines an automorphism of the root datum
\[
(M,M^\vee, \Roots, \Roots^\vee)
\]
attached to $\bG$, and by~\cite[Exp.~XXIII, Th\'eor\`eme~4.1]{sga33} this procedure identifies 
$\mathrm{Aut}(\bG,\bT, M, \Roots, \Delta, X)$
with the group of automorphisms of the root datum $(M,M^\vee, \Roots, \Roots^\vee)$ stabilizing the subset $\Delta \subset \Roots$. We will say that a group $A$ acts on $\bG$ \emph{by pinned automorphisms} if it acts via a group homomorphism
\[
A \to \mathrm{Aut}(\bG,\bT, M, \Roots, \Delta, X).
\]
In this situation, there is an induced action of $A$ on $M$ preserving $\Roots$ and $\Roots_+$.

\subsection{Root subgroups, Borel subgroup, and unipotent subgroup}
\label{ss:root-subgps}

Recall that for any root $\gamma \in \Roots$, there is a closed immersion
\begin{equation}\label{eqn:root-subgp}
\exp_\gamma : \Lie(\bG)^\gamma \hookrightarrow \bG
\end{equation}
whose image, denoted by $\bU_\gamma$, is a closed subgroup scheme of $\bG$: see~\cite[Exp.~XXII, Th\'eor\`eme~1.1]{sga33}. 

We endow the subset $\Roots_+ \subset \Roots$ with some arbitrary order. Then one can consider the product morphism
\[
 \prod_{\alpha \in \Roots_+} \bU_\alpha \to \bG,
\]
where the product on the left-hand side (to be understood as fiber product over $S$) is taken with respect to our chosen order on $\Roots_+$.  This morphism is a closed immersion, and its image $\bU$ is a subgroup scheme which does not depend on the choice of order on $\Roots_+$; see~\cite[Exp.~XXII, \S 5.5]{sga33}. The product morphism
\[
 \bT \ltimes_S \bU \to \bG
\]
is also a closed immersion, and a homomorphism of group schemes; its image will be denoted $\bB$. On geometric fibers of $\bG \to S$, $\bB$ is a Borel subgroup in the usual sense, and $\bU$ is its unipotent radical. Similar considerations using the negative roots $-\Roots_+$ produce a closed subgroup scheme denoted $\bU^-$.

\subsection{Chevalley systems}
\label{ss:chevalley}

Recall that the set $X = (X_\alpha: \alpha \in \Delta)$ is indexed by the simple roots.  A \emph{Chevalley system} is a collection $(Y_\alpha : \alpha \in \Roots)$ parametrized by $\Roots$, where each $Y_\alpha$ is again a nowhere vanishing section of $\Lie(\bG)^\alpha$, and where the entire collection is subject to certain conditions, spelled out in~\cite[Exp.~XXIII, D\'efinition~6.1]{sga33}.
By~\cite[Exp.~XXIII, Proposition~6.2]{sga33}, a Chevalley system exists. 
More specifically, the proof of this proposition shows that there exist Chevalley systems which satisfy the following additional conditions:
\begin{itemize}
\item
for any $\alpha \in \Delta$ we have $X_\alpha = Y_\alpha$;
\item
for any $\alpha \in \Roots_+$ the sections $Y_\alpha$ and $Y_{-\alpha}$ are dual to each other with respect to the pairing of~\cite[Exp.~XX, Corollaire~2.6]{sga33}.
\end{itemize}
When considering Chevalley systems below we will always tacitly assume (following e.g.~the conventions in~\cite[\S3.2.2]{BT84}) that these additional conditions are satisfied.
In this case, one can use the notation $(X_\alpha : \alpha \in \Roots)$ for a Chevalley system, and this system is determined by the ``positive'' subset $(X_\alpha : \alpha \in \Roots_+)$.

According to~\cite[Exp.~XXIII, Corollaire~6.5]{sga33}, one consequence of the conditions in the definition is that in a Chevalley system, if $\alpha$ and $\beta$ are roots such that $\alpha+\beta$ is also a root, then
\begin{equation}
\label{eqn:chevalley-rule}
[X_\alpha, X_\beta] = \pm rX_{\alpha+\beta}
\quad
\begin{array}{c}
\text{where $r \in \{1,2,3\}$ is the smallest positive} \\
\text{integer such that $\beta - r\alpha$ is not a root.}
\end{array}
\end{equation}
Of course, if $\alpha + \beta$ is not a root then $[X_\alpha, X_\beta]=0$.

\begin{ex}\label{ex:D4-chevalley}
Let us give an example where one can write down (the positive part of) a Chevalley system explicitly in terms of a pinning $(X_\alpha : \alpha \in \Delta)$.
Assume that $\bG$ is of type $\mathsf{D}_4$, and number the simple roots as $\alpha_1, \alpha_2, \alpha_3, \alpha_4$ with $\langle \alpha_2, \alpha_i^\vee \rangle=-1$ for $i \in \{1,3,4\}$. There are eight other positive roots.  One can easily check that the following vectors are the positive part of a Chevalley system:
\[
\hbox{\footnotesize$
\begin{aligned}
X_{\alpha_i + \alpha_2} &= [X_{\alpha_i}, X_{\alpha_2}] \quad(i = 1,2,3), &
  X_{\alpha_2+\alpha_3+\alpha_4} &= [[X_{\alpha_3}, X_{\alpha_2}], X_{\alpha_4}], \\
X_{\alpha_1+\alpha_2+\alpha_3} &= [[X_{\alpha_1}, X_{\alpha_2}], X_{\alpha_3}], &
  X_{\alpha_1+\alpha_2+\alpha_3+\alpha_4} &= [[[X_{\alpha_1},X_{\alpha_2}], X_{\alpha_3}], X_{\alpha_4}], \\
X_{\alpha_1+\alpha_2+\alpha_4} &= [[X_{\alpha_1}, X_{\alpha_2}], X_{\alpha_4}], &
  X_{\alpha_1+2\alpha_2+\alpha_3+\alpha_4} &= [[[[X_{\alpha_1},X_{\alpha_2}], X_{\alpha_3}], X_{\alpha_4}], X_{\alpha_2}].
\end{aligned}$}
\]
\end{ex}

Once a Chevalley system $(X_\alpha : \alpha \in \Roots_+)$ as above is fixed, for any $\alpha \in \Roots_+$ there exists a unique morphism of $S$-group schemes
\begin{equation}
\label{eqn:sl2-maps}
\varphi_\alpha : \mathrm{SL}_{2,S} \to \bG
\end{equation}
which satisfies
\[
\varphi_\alpha \begin{pmatrix} 1 & a \\ 0 & 1 \end{pmatrix} = \exp_\alpha( a X_\alpha), \qquad \varphi_\alpha \begin{pmatrix} 1 & 0 \\ a & 1 \end{pmatrix} = \exp_{-\alpha}( a X_{-\alpha})
\]
for any $a \in \bG_{\mathrm{a},S}$,
see~\cite[Exp.~XX, Corollaire~2.6]{sga33}.
In this case we automatically have
\[
\varphi_\alpha \begin{pmatrix} a & 0 \\ 0 & a^{-1} \end{pmatrix} = \alpha^\vee(a)
\]
for any $a \in \bG_{\mathrm{m},S}$. We set
\[
n_\alpha := \varphi_\alpha \begin{pmatrix} 0 & 1 \\ -1 & 0 \end{pmatrix}.
\]
(This section coincides with the section denoted $w_\alpha(\mathrm{X}_\alpha)$ in~\cite[Exp.~XXIII, D\'efinition~6.1]{sga33}.)

\subsection{Reduction to simply connected quasi-simple groups}
\label{ss:reduction-scqs}

We will say that a pinned reductive group scheme $(\bG,\bT, M, \Roots, \Delta, X)$ is quasi-simple and simply connected if $\Roots$ is indecomposable and $\Roots^\vee$ generates $M^\vee$ (or, in other words, if $\bG_{\overline{s}}$ is quasi-simple and simply connected in the usual sense of semisimple algebraic groups for any geometric point $\overline{s}$ of $S$).
For some constructions below, we will reduce the problem to this case (or to products of such groups) as follows.

The root datum of our group $\bG$ with respect to $\bT$ is $(M,M^\vee, \Roots, \Roots^\vee)$. We set $M_\sconn:=\Hom_{\Z}(\Z \Roots^\vee, \Z)$, and $M_\sconn^{\vee}:= \Z \Roots^\vee$. We have natural (dual) maps $M_\sconn^{\vee} \to M^\vee$ and $M \to M_\sconn$. The latter morphism is injective on the subset $\Roots$, which can therefore also be regarded as a subset of $M_\sconn$. We then have a morphism of root data
\[
(M_\sconn,M_\sconn^\vee, \Roots, \Roots^\vee) \to (M,M^\vee, \Roots, \Roots^\vee)
\]
in the sense of~\cite[Exp.~XXI, D\'efinition~6.1.1]{sga33}.
If we denote by
\[
(\bG_\sconn, \bT_\sconn, M_\sconn, \Roots, \Delta, X_\sconn)
\]
the pinned reductive group scheme over $S$ with root datum $(M_\sconn,M_\sconn^\vee, \Roots, \Roots^\vee)$, then this morphism corresponds to a morphism of pinned groups
\[
\bG_\sconn \to \bG.
\]
Moreover, for any geometric point $\overline{s}$ of $S$, this morphism identifies $(\bG_\sconn)_{\overline{s}}$ with the simply connected cover of the derived subgroup of $\bG_{\overline{s}}$. The roots of $\bG$ and $\bG_\sconn$ are the same, and our morphism $\bG_\sconn \to \bG$ restricts to an isomorphism on root subgroups with the same label, identifying the given pinnings.

One can make $\bG_\sconn$ more concrete as follows. 
Consider the decomposition
\[
\Roots = \bigsqcup_{i \in I} \Roots_i
\]
of $\Roots$ as a direct sum of indecomposable constituents. This determines a direct-sum decomposition 
\[
M_\sconn = \bigoplus_{i \in I} M_{\sconn,i},
\]
and hence a product decomposition (where the product is fiber product over $S$)
\begin{equation}
\label{eqn:decomp-H}
\bG_\sconn = \prod_{i \in I} \bG_{\sconn,i}
\end{equation}
where $\bG_{\sconn,i}$ is the pinned reductive group scheme over $S$ with root datum
\[
(M_{\sconn,i},(M_{\sconn,i})^\vee, \Roots_i, \Roots^\vee_i).
\]

\section{Equivalence classes of roots and twisted \texorpdfstring{$\mathrm{SL}_2$}{SL2}-maps}
\label{sec:equiv-roots}

\subsection{An equivalence relation on roots}
\label{ss:equiv-relation}

In this subsection we consider an arbitrary reduced root system $\Roots$ in a real vector space $V$, and a basis $\Delta \subset \Roots$ of $\Roots$. We assume we are given a group $A$ and an action of $A$ on $V$ preserving $\Roots$ and $\Delta$. We will denote by $\Roots_+$ the system of positive roots determined by $\Delta$. (This subset is also preserved by $A$.) Let us consider the equivalence relation $\sim$ on $\Roots_+$ defined as follows:
\begin{equation}
\label{eqn:roots-equiv}
\alpha \sim \beta \quad\text{if $
\sum_{\gamma \in A \cdot \alpha} \gamma $ and 
$\sum_{\delta \in A \cdot \beta} \delta$ are scalar multiples of one another.}
\end{equation}
(Here we mean scalar multiples
in the real vector space $V$. This relation of course depends on the group $A$ and its action on $\Roots$, although it does not appear in the notation.)

First, assume that $\Roots$ is indecomposable. In this case this equivalence relation is studied in~\cite{steinberg}, and this analysis (based on a case-by-case verification) shows that
each equivalence class $E$ for $\sim$ is of one of the following two forms.
\begin{enumerate}
\item[(i)] $E$ is a single $A$-orbit, and if $\alpha, \beta \in E$ then
$\alpha+\beta$ is not a root.
\item[(ii)] $E$ is of the form $\{ \alpha, a \cdot \alpha, \alpha + a \cdot \alpha \}$ for some $a \in A$ and $\alpha \in \Roots$; in this case, $\{ \alpha, a \cdot \alpha \}$ is an $A$-orbit, and $\alpha + a\cdot \alpha$ is fixed by $A$.
\end{enumerate}
In more detail, if the image of $A$ in the automorphism group of the Dynkin diagram of $\Roots$ is a cyclic group, then this statement is~\cite[Claim (2$'$) in the proof of Theorem~8.2]{steinberg}. The only case not covered is that in which $\Roots$ is of type $\mathsf{D}_4$, and the image of $A$ is the full automorphism group (which is the symmetric group $\mathfrak{S}_3$).  In that case, a direct calculation shows that there are six equivalence classes for $\sim$, each of type (i).

As explained in~\cite{steinberg}, equivalence classes of type (ii) occur if and only if $\Roots$ is of type $\mathsf{A}_{2n}$ and $A$ acts nontrivially. More explicitly, if we choose a labeling $\alpha_1, \ldots, \alpha_{2n}$ of the simple roots such that $\langle \alpha_i, \alpha_j^\vee \rangle = -1$ if $|j-i|=1$, then the equivalence classes of type (ii) are the sets of the form
\begin{equation}
\label{eqn:subsets-A2n}
\{ \alpha_i + \cdots + \alpha_n, \quad \alpha_{n+1} + \cdots + \alpha_{2n+1-i}, \quad \alpha_i + \cdots + \alpha_{2n+1-i} \}
\end{equation}
with $i \in \{1, \dots, n\}$. 

Now we consider the general case. Write 
\begin{equation}
\label{eqn:decomposition-Roots-indec}
\Roots = \bigsqcup_{i \in I} \Roots_i
\end{equation}
for the decomposition of $\Roots$ as a direct sum of its indecomposable constituents. Given $a \in A$ and $i \in I$ there exists a unique $\sigma(a)(i) \in I$ such that $a(\Roots_i)=\Roots_{\sigma(a)(i)}$, and this operation defines a group homomorphism
\begin{equation}
\label{eqn:action-A-I}
\sigma : A \to \mathfrak{S}_I.
\end{equation}
For any $i \in I$, we denote by $V_i \subset V$ the subspace spanned by $\Roots_i$, by $A_i = \{a \in A \mid \sigma(a)(i)=i\}$ the stabilizer of the component $\Roots_i$, and by $\sim_i$ the equivalence relation defined as above for the action of $A_i$ on the root system $\Roots_i$.

\begin{lem}
If $\alpha,\beta \in \Roots_+$, we have $\alpha \sim \beta$ if and only if there exist $a \in A$ and $i \in I$ such that $a \cdot \alpha$ and $\beta$ both belong to $\Roots_i$, and moreover $(a \cdot \alpha) \sim_i \beta$.
\end{lem}

\begin{proof}
Let us write $i \leftrightarrow j$ if $i$ and $j$ belong to the same orbit of the action of $A$ on $I$ via $\sigma$.
In this case, fix an element $a_{i,j} \in A$ such that $a_{i,j}(\Roots_i)=\Roots_j$. Then if $\alpha \in \Roots_i$, the projection of $\sum_{a \in A} a \cdot \alpha$ to $V_j$ along the decomposition $V=\bigoplus_k V_k$ is  $a_{i,j} \cdot (\sum_{a \in A_i} a \cdot \alpha)$ if $i \leftrightarrow j$, and zero otherwise. The claim easily follows.
\end{proof}

This lemma and the analysis of the indecomposable case above show that, in general, each equivalence class $E$ for $\sim$ is of one of the following two forms.
\begin{enumerate}
\item[(i$'$)] $E$ is a single $A$-orbit, and if $\alpha, \beta \in A$ then
$\alpha+\beta$ is not a root.
\item[(ii$'$)] $\Roots$ contains an indecomposable constituent $\Roots'$ of type $\mathsf{A}_{2n}$ whose stabilizer $A'$ acts nontrivially on this constituent, and
\[
E = \bigsqcup_{a \in A/A'} a \cdot E'
\]
where $E'$ is a subset of the form~\eqref{eqn:subsets-A2n} in the given indecomposable constituent of $\Roots$.
\end{enumerate}

We will say that a positive root $\beta$ is \emph{special} if it belongs to an equivalence class $E$ of type (ii$'$) and is a sum of two roots which belong to $E$. In other words,
if $E$ is an equivalence class of type (ii$'$) and $E'$ is as above, then $E'$ is of the form $\{\alpha, a \cdot \alpha, \alpha + a \cdot \alpha\}$ for some $\alpha \in \Roots'$ and $a \in A'$; the special roots in $E$ are those of the form $b \cdot (\alpha + a \cdot \alpha)$ for some $b \in A$.  In this situation, $E$ is the union of exactly two $A$-orbits, one consisting of special roots, and the other of nonspecial roots.

\begin{rmk}
\label{rmk:equiv-closed}
Note that each equivalence class $E$ for $\sim$ is ``closed under positive combinations'' in the following sense: if $\alpha, \beta \in E$ are distinct and $i,j \in \Z_{\geq 1}$ are such that $i\alpha+j\beta \in \Roots$, then in fact $i\alpha +j\beta \in E$.  
If $E$ is of type (ii$'$), then this claim is clear from the explicit description above. If $E$ is of type (i$'$), 
we use the fact that if $\alpha,\beta$ are positive roots such that $\alpha + \beta \notin \Roots$, then
there are no $i,j \in \Z_{\geq 1}$ and such that $i\alpha+j\beta \in \Roots$. (In fact, since neither $\alpha+\beta$ nor $\alpha-\beta$ are roots we have $\langle \beta,\alpha^\vee \rangle$=0 by~\cite[Chap.~VI, Corollary to Theorem~1]{bourbaki}. If we assume for a contradiction that there exists a root of the form $i\alpha+j\beta$ with $i,j \geq 1$, and choose these coefficients such that $i+j$ is minimal, then since $\langle i\alpha+j\beta, \alpha^\vee \rangle = 2i>0$, by~\cite[Chap.~VI, Corollary to Theorem~1]{bourbaki} $(i-1)\alpha+j\beta$ is a root. By minimality we must have $i=1$, so that $j\beta$ is a root. Since $\Roots$ is reduced this forces $j=1$, which provides a contradiction since $\alpha+\beta$ is not a root.)
\end{rmk}

In~\S\ref{ss:fixed-pts-alg-closed-fields} we will use the following fact.

\begin{lem}
\label{lem:roots-coinvariants}
Let $(\alpha_1, \dots, \alpha_r)$ be representatives for the equivalence classes for $\sim$, and let $(\bar{\alpha}_1, \dots, \bar{\alpha}_r)$ be their images in the coinvariants $V_A$. Then each $\bar{\alpha}_i$ is nonzero, and the lines $(\R \cdot \bar{\alpha}_1, \dots, \R \cdot \bar{\alpha}_r)$ in $V_E$ are pairwise distinct.
\end{lem}

\begin{proof}
The proof can be easily reduced to the case $\Roots$ is irreducible. In this case, the claim can be checked case-by-case by inspecting the description of roots in each type (see e.g.~\cite{bourbaki}), and the various possibilities for the action of $A$.
\end{proof}

\subsection{Chevalley--Steinberg systems}

From now on in this section we fix a pinned reductive group scheme~\eqref{eqn:pinned-gp}
over $S$ as in~\S\ref{ss:pinned-red-groups} (without any special assumption on $S$), endowed with an action of a group $A$ by pinned automorphisms (see~\S\ref{ss:pinned-red-groups-aut}).
Below we will need to consider Chevalley systems which afford some compatibility property with our given action of $A$. (Such systems are called \emph{Chevalley--Steinberg systems} in~\cite[Definition~4.3]{landvogt} or~\cite[D\'efinition~2.1.4]{lourenco}.) The statement of this property involves the equivalence relation $\sim$ on $\Roots_+$ from~\S\ref{ss:equiv-relation} for the given action of $A$ on $\Roots$, seen as a root system in the subspace it generates in $\R \otimes_{\Z} M$.

\begin{prop}
\label{prop:chevalley-steinberg}
There exists a Chevalley system $(X_\alpha : \alpha \in \Roots)$ such that 
\begin{equation}
\label{eqn:action-A-root-vectors}
a \cdot X_\beta = X_{a \cdot \beta}
\end{equation}
for all $a \in A$ and all nonspecial $\beta \in \Roots_+$.
\end{prop}

\begin{proof}
First, assume that $\Roots$ is indecomposable. In this case, the statement is essentially proved in~\cite[Proposition~4.4]{landvogt}. Let us give a brief outline of the argument, for the reader's convenience. The basic observation is that making sign changes in a Chevalley system on vectors associated with nonsimple positive roots (and the corresponding changes for the opposite roots) produces a new Chevalley system.

We can of course replace $A$ by its image in $\mathrm{Aut}(\bG,\bT, M, \Roots, \Delta, X)$.  Since $\bG$ is quasi-simple, the group $\mathrm{Aut}(\bG,\bT, M, \Roots, \Delta, X)$ is the automorphism group of a connected Dynkin diagram, and therefore has order $1$, $2$, or $6$ (the last possibility occurring only in type $\mathsf{D}_4$, where the group is $\mathfrak{S}_3$).  
The subgroup $A$ thus is cyclic of order $1$, $2$, $3$, or $A\cong \mathfrak{S}_3$.  One can then make sign changes on some elements $X_\alpha$ ($\alpha \in \Roots_+$) to ensure that the desired condition holds, by considering the various orbits $A \cdot \beta$ of nonspecial roots as follows (see~\cite[Proposition~4.4]{landvogt} for details).
\begin{itemize}
\item 
If $A$ is cyclic (i.e., of order $1$, $2$, or $3$) and the stabilizer in $A$ of $\beta$ has odd order (either $1$ or $3$), then one can ensure that
$a \cdot X_\gamma = X_{a \cdot \gamma}$ for all $a \in A$ and $\gamma \in A \cdot \beta$.
\item 
Assume now that $A$ has order $2$ and the stabilizer of $\beta$ is all of $A$ (so that $A \cdot \beta = \{\beta\}$). Denote by $a$ the unique nontrivial element in $A$. Since $\beta$ is not special, it is not of the form $\gamma + a \cdot \gamma$ for some $\gamma \in \Roots_+$. Then a brief calculation using root system combinatorics shows that $a \cdot X_\beta = X_{\beta}$ for any $a \in A$. 
\item When $A$ has order $6$, the argument in~\cite[Proposition~4.4]{landvogt} does not seem to be quite complete.  In this case, $\bG$ has type $\mathsf{D}_4$, and all equivalence classes for $\sim$ are of type (i). Applying $a$ to each of the equations in Example~\ref{ex:D4-chevalley} one sees that this Chevalley system satisfies
\[
a \cdot X_\beta = X_{a \cdot \beta} \qquad\text{for all $a \in A$.}
\]
(In fact, it is enough to check this for $a$ belonging to some set of generators of the group $A \cong \mathfrak{S}_3$.)
\end{itemize}

Now, let us explain how to treat the general case. Using the construction of~\S\ref{ss:reduction-scqs} one can assume that $\bG$ is a product of quasi-simple simply connected groups. (In fact, if $\bG_\sconn$ is as in~\S\ref{ss:reduction-scqs}, the root subgroups and the notion of Chevalley system coincide for $\bG_\sconn$ and $\bG$.) In this case, associated with the decomposition~\eqref{eqn:decomposition-Roots-indec} is the decomposition
\[
\bG = \prod_{i \in I} \bG_i.
\]
Recall also the action of $A$ on $I$ given by $\sigma$, see~\eqref{eqn:action-A-I}, and choose a subset $J \subset I$ of representatives for the $A$-orbits. For any $j \in J$, we also denote by $A_j \subset A$ the stabilizer of $j$, and choose representatives $a^j_1, \ldots, a^j_{r(j)}$ for the quotient $A/A_j$. We then have a bijection
\[
\bigsqcup_{j \in J} \Roots_j \times \{1, \ldots, r(j)\} \simto \Roots
\]
given by $(\alpha, k) \mapsto a^j_k \cdot \alpha$ if $\alpha \in \Roots_j$, which restricts to a bijection between nonspecial roots on each side. We choose a Chevalley system as in the proposition for each $\bG_j$ ($j \in J$), and then
set
\[
X_{a^j_k \cdot \alpha} = a^j_k \cdot X_\alpha
\]
if $\alpha \in \Roots_j$. It is easily checked that this procedure produces a Chevalley system with the required property.
\end{proof}

\begin{rmk}
In Proposition~\ref{prop:chevalley-steinberg} we have specified only what happens for nonspecial roots. In fact, as soon as special roots exist, there does not exist any Chevalley system which satisfies~\eqref{eqn:action-A-root-vectors} for all $a \in A$ and $\alpha \in \Roots$.
More precisely, this equality sometimes fails by a sign when $\beta$ is special. (It is instructive to work out this failure in the example of the action of $\Z/2\Z$ on $\mathrm{SL}_{3,\Z}$ considered in~\S\ref{sss:sl3_action}.). It is possible to specify these signs explicitly, but this require some choices. Since this will not be necessary below, we omit the details.
\end{rmk}

\subsection{Twisted \texorpdfstring{$\mathrm{SL}_2$}{SL2}-maps}
\label{ss:SL2-maps}

From now on we fix a Chevalley system $(X_\alpha : \alpha \in \Roots)$ as in Proposition~\ref{prop:chevalley-steinberg}.
Our goal in the rest of this section is to explain the construction of analogues of the maps~\eqref{eqn:sl2-maps} for the group $\bG^A$. More specifically, these maps will be
attached to equivalence classes for $\sim$. In case $E$ is of type (i$'$), it will be given by a group scheme morphism
\[
\imath_E : \mathrm{SL}_{2,S} \to \bG^A.
\]
In case $E$ is of type (ii$'$), it will be given by a group scheme morphism
\[
\imath_E : (\mathrm{SL}_{3,S})^{\Z/2\Z} \to \bG^A,
\]
where the action on the left-hand side is that considered in~\S\ref{sss:sl3_action}. For orbits of type (i$'$) the construction is uniform, but depends on the choice of Chevalley system. For orbits of type (ii$'$) the construction is more ad hoc.

\subsubsection{Equivalence classes of type (i$'$)}
\label{sss:SL2maps-i}

We start with the easier case when $E$ is an equivalence class of type (i$'$). (In particular, this class consists of nonspecial roots.) For any $\alpha \in E$ we have a morphism $\varphi_\alpha$ as in~\eqref{eqn:sl2-maps}. Moreover, if $a \in A$ and $\alpha \in E$, since $a \cdot X_\alpha = X_{a \cdot \alpha}$, by uniqueness we have
\[
a \circ \varphi_\alpha = \varphi_{a \cdot \alpha}
\]
(where by abuse we denote by $a$ the action of $a$ on $\bG$).
By the comments in Remark~\ref{rmk:equiv-closed}, if $\alpha, \alpha' \in E$ are distinct then no positive combination of $\alpha$ and $\alpha'$ is a root, which implies that $\bU_\alpha$ and $\bU_{\alpha'}$ commute by the commutation relations~\cite[Exp.~XXII, Corollaire~5.5.2]{sga33}. As a consequence, the map
\[
\imath'_E := \prod_{\alpha \in E} \varphi_\alpha : \prod_{\alpha \in E}\mathrm{SL}_{2,S}\to \bG
\]
is a morphism of group schemes, which is moreover $A$-equivariant where $A$ acts on the left-hand side by permuting the factors (according to the action on $E$). Passing to $A$-fixed points and using the fact that
\[
\left( \prod_{\alpha \in E}\mathrm{SL}_{2,S} \right)^A = \mathrm{SL}_{2,S},
\]
we obtain the desired morphism $\imath_E$.

An important property of this morphism, which will be used below, is the following. The considerations above show that we have a natural closed immersion of group schemes
\[
\prod_{\alpha \in E} \bU_\alpha \to \bG.
\]
(Because the $\bU_\alpha$'s appearing here commute with one another, we do not need to specify an order on $E$.) The image of this morphism is stable under the action of $A$, and will be denoted $\bU_E$. If we denote by $\mathrm{U}_{2,S}$ the subgroup of $\mathrm{SL}_{2,S}$ of upper triangular unipotent matrices, then $\imath_E$ restricts to an isomorphism
\begin{equation}
\label{eqn:root-subgp-type-i}
\mathrm{U}_{2,S} \simto (\bU_E)^A.
\end{equation}
Moreover, we have
\begin{equation}\label{eqn:root-subgp-type-i-torus}
\imath_E \begin{pmatrix} a & 0 \\ 0 & a^{-1} \end{pmatrix} = \prod_{\alpha \in E} \alpha^\vee(a).
\end{equation}

\subsubsection{Equivalence classes of type (ii$'$)}
\label{sss:SL2maps-ii}

Now, let us fix an equivalence class $E$ of type (ii$'$). In this case, we use the considerations of~\S\ref{ss:reduction-scqs} to reduce the construction to the case where $\bG$ is a product of quasi-simple simply connected groups.
Namely, consider the group $\bG_\sconn$ constructed in~\S\ref{ss:reduction-scqs}.
The action of $A$ on $\bG$ induces an action on its root datum, which in turn provides an action on the root datum of $\bG_\sconn$, and finally an action on $\bG_\sconn$ by pinned group automorphisms. For this action, the morphism $\bG_\sconn \to \bG$ is $A$-equivariant, and hence restricts to a group scheme morphism $(\bG_\sconn)^A \to \bG^A$. The roots associated with these two group schemes coincide, as does the equivalence relations on $\Roots_+$ determined by the actions of $A$ (see~\S\ref{ss:equiv-relation}) and the notions of orbits of type (i$'$) or (ii$'$). It therefore suffices to construct our morphism for the group $\bG_\sconn$. 

Recall next the decomposition~\eqref{eqn:decomp-H} and the action of $A$ on $I$ determined by~\eqref{eqn:action-A-I}.
If we decompose $I$ into its orbits for this action:
\[
I = \bigsqcup_{j \in J} I_j
\]
and set $\bG_\sconn^j = \prod_{i \in I_j} \bG_{\sconn,i}$, then each $\bG_\sconn^j$ is stable under the action of $A$, and moreover by Lemma~\ref{lem:fixed-pts-properties}\eqref{it:fixed-pts-product} we have
\[
(\bG_\sconn)^A = \prod_{j \in J} (\bG_\sconn^j)^A.
\]
We can (and will) therefore assume that the action of $A$ on $I$ is transitive.

In this case, all constituents $\Roots_i$ have the same type. Since we are interested in equivalence classes of type (ii$'$), it suffices to consider the case when this type is $\mathsf{A}_{2n}$ for some $n \geq 1$. We can also assume that the stabilizer in $A$ of each component acts nontrivially on this component. For any $i \in I$, the intersection $\Delta \cap \Roots_i$ is a basis of $\Roots_i$. Let us fix an identification of $\fR_i$ and its basis $\Delta_i$ with the root system of $\mathrm{SL}_{2n+1}$ and its basis from Example~\ref{ex:roots-A2n}. (There exist exactly two such identifications; we choose one of them.) This determines an identification of pinned reductive group schemes
\[
\bG_{\sconn,i} = \mathrm{SL}_{2n+1,S}
\]
(where the pinning on the right-hand side is as in Example~\ref{ex:roots-A2n}). Taking these identifications together we obtain an identification
\[
\bG_\sconn = \bigl( \mathrm{SL}_{2n+1,S} \bigr)^{\times_S I}.
\]

The group of pinned automorphisms of $( \mathrm{SL}_{2n+1,S} )^{\times_S I}$ is $(\Z/2\Z)^{\times I} \rtimes \mathfrak{S}_I$; our given action of $A$ is therefore determined by a ``lift'' of~\eqref{eqn:action-A-I} to a group homomorphism $\tilde{\sigma} : A \to (\Z/2\Z)^{\times I} \rtimes \mathfrak{S}_I$. There exists $k \in \{1, \ldots, n\}$ such that our given equivalence class $E$ of type (ii$'$) is the union of the subsets
\[
\{ \alpha_k + \cdots + \alpha_n, \quad \alpha_{n+1} + \cdots + \alpha_{2n+1-k}, \quad \alpha_k + \cdots + \alpha_{2n+1-k} \}
\]
in the root system of each copy of $\mathrm{SL}_{2n+1,S}$. One can then consider the morphism
\[
 (f_{k,n})^{\times_S I} : ( \mathrm{SL}_{3,S} )^{\times_S I} \to ( \mathrm{SL}_{2n+1,S} )^{\times_S I},
\]
see~\eqref{eqn:tilde-morph-SL3}. This morphism is equivariant with respect to the actions of $(\Z/2\Z)^{\times I} \rtimes \mathfrak{S}_I$ on each side, so it induces a morphism
\[
 \bigl( ( \mathrm{SL}_{3,S} )^{\times_S I} \bigr)^{(\Z/2\Z)^{\times I} \rtimes \mathfrak{S}_I} \to \bigl( ( \mathrm{SL}_{2n+1,S} )^{\times_S I} \bigr)^{(\Z/2\Z)^{\times I} \rtimes \mathfrak{S}_I}
\]
between fixed points. Here the left-hand side identifies with $(\mathrm{SL}_{3,S})^{\Z/2\Z}$, and our given morphism $\tilde{\sigma}$ determines a morphism
\[
\bigl( ( \mathrm{SL}_{2n+1,S} )^{\times_S I} \bigr)^{(\Z/2\Z)^{\times I} \rtimes \mathfrak{S}_I} \to (\bG_\sconn)^A.
\]
This construction therefore provides the desired morphism
\[
\imath_E : (\mathrm{SL}_{3,S})^{\Z/2\Z} \to (\bG_\sconn)^A.
\]

\begin{rmk}
In the construction above we have chosen an identification of each $\Roots_i$ with the root system of type $\mathsf{A}_{2n}$. Changing these identifications amounts to composing the identification $\bG_\sconn = \bigl( \mathrm{SL}_{2n+1,S} \bigr)^{\times_S I}$ with the action of a certain element of $(\Z/2\Z)^{\times I}$. This change will lead to a different morphism $( \mathrm{SL}_{3,S} )^{\times_S I} \to \bG_\sconn$, but it will not affect its restriction $(\mathrm{SL}_{3,S})^{\Z/2\Z} \to (\bG_\sconn)^A$. In other words, the morphism $\imath_E$ does not depend on these choices.
\end{rmk}

The morphism we have constructed here again has a property similar to that explained at the end of~\S\ref{sss:SL2maps-i}. Namely, choosing any order on $E$ we can consider the product morphism
\[
\prod_{\alpha \in E} \bU_\alpha \to \bG.
\]
This morphism is a closed immersion, and its image is a subgroup scheme which does not depend on the choice of order; it will be denoted $\bU_E$. In particular, this image is stable under the action of $A$. Recall the subgroup scheme $\mathrm{U}_{3,S}$ of $\mathrm{SL}_{3,S}$ considered in~\S\ref{sss:sl3_action}. Then the morphism $\imath_E$ restricts to an isomorphism
\begin{equation}
\label{eqn:root-subgp-type-ii}
(\mathrm{U}_{3,S})^{\Z/2\Z} \simto (\bU_E)^A.
\end{equation}
To check this claim, one can e.g.~use the following fact.

\begin{lem}
\label{lem:product-gps}
Let $\bH$ be a pinned reductive group scheme over $S$ which is a quasi-simple simply connected group, and assume given a finite set $I$ and an action of a group $A$ on $\bH^{\times_S I}$ by pinned automorphisms. As in~\eqref{eqn:action-A-I} this determines an action of $A$ on $I$. If this action is transitive, then, for any $i \in I$, projection onto the component parametrized by $i$ induces an isomorphism of group schemes
\[
\bigl( \bH^{\times_S I} \bigr)^A \simto \bH^{A_i}
\]
where $A_i$ is the stabilizer of $i$ in $A$.
\end{lem}

\begin{proof}
The inverse isomorphism is constructed as follows. By assumption the action of $A$ on $I$ induces a bijection $A/A_i \simto I$. Choose, for any $j \in I$, an element $a_j \in A$ such that $\sigma(a_j)(i)=j$. Then the action of $a_j$ on $\bH^{\times_S I}$ identifies the copy of $\bH$ indexed by $i$ with that indexed by $j$. It is easily checked that the assignment
\[
h \mapsto \prod_j (a_j \cdot h)
\]
induces a morphism $\bH^{A_i} \to \bigl( \bH^{\times_S I} \bigr)^A$ which is inverse to the morphism of the statement.
\end{proof}

Finally, we have a counterpart of~\eqref{eqn:root-subgp-type-i-torus}:
\begin{equation}\label{eqn:root-subgp-type-ii-torus}
\imath_E\begin{pmatrix}
a & 0 & 0 \\
0 & 1 & 0 \\
0 & 0 & a^{-1}
\end{pmatrix} = \prod_{\substack{\alpha \in E \\ \text{$\alpha$ special}}} \alpha^\vee(a).
\end{equation}

\section{Flatness and smoothness}
\label{sec:flatness-smoothness}

\subsection{Statement}

We 
continue with our pinned reductive group scheme~\eqref{eqn:pinned-gp}
over $S$, and our group $A$ which acts on $\bG$ by pinned automorphisms.
Let $M_A$ denote the coinvariants for the induced action of $A$ on $M$ (see~\S\ref{ss:ex-diag-gps}). Recall also that $A$ acts on $\fR$, and permutes its indecomposable components.

The following statement is the main result of this paper. 

\begin{thm}
\phantomsection
\label{thm:fixed-pts}
 \begin{enumerate}
  \item 
  \label{it:fixed-pts-1}
  The group scheme $\bG^A$ is flat over $S$.
  \item 
  \label{it:fixed-pts-2}
  The group scheme $\bG^A$ has geometrically connected fibers over $S$ iff either $M_A$ is torsion-free or 
  $S$ has exactly one residual characteristic $\ell>0$ and the torsion subgroup of $M_A$ is an $\ell$-group.
  \item 
  \label{it:fixed-pts-3}
  The group scheme $\bG^A$ is smooth over $S$ iff the following conditions hold:
  \begin{itemize}
  \item the order of the torsion subgroup of $M_A$ is coprime to all residual characteristics of $S$;
  \item if $\Roots$ has an indecomposable component of type $\mathsf{A}_{2n}$ for some $n \geq 1$ whose stabilizer in $A$ acts nontrivially on this component, then $2$ is not a residual characteristic of $S$.
  \end{itemize}
  \item 
  \label{it:fixed-pts-5}
  If $S$ is the spectrum of a mixed characteristic
  DVR and $M_A$ is torsion-free, then $\bG^{A}$ is a quasi-reductive $S$-group scheme in the sense of Prasad--Yu~\cite{py}.
 \end{enumerate}
\end{thm}

\subsection{Fixed points in the big cell}
\label{ss:big-cell}

The main step in our proof of Theorem~\ref{thm:fixed-pts} will be the study of the fixed points of $A$ on the ``big cell'' in $\bG$. Recall the subgroups $\bU$, $\bU^-$ introduced in~\S\ref{ss:root-subgps}.
For any $\alpha \in \Roots$ and $a \in A$, the action of $a$ on $\bG$ induces an isomorphism
\begin{equation*}
\bU_\alpha \simto \bU_{a \cdot \alpha}.
\end{equation*}
By the independence of $\bU$ on the choice of order on $\Roots_+$, this implies that this subgroup is stable under the action of $A$. Similar comments apply to $\bU^-$.

By~\cite[Exp.~XXII, Proposition~4.1.2]{sga33}, the multiplication morphism
\[
\bU^- \times_S \bT \times_S \bU \to \bG
\]
is an open immersion. Its image (called the \emph{big cell}) is denoted $\bC \subset \bG$. 
This open subscheme is stable under the action of $A$, and using Lemma~\ref{lem:fixed-pts-properties}\eqref{it:fixed-pts-product} we see that multiplication induces an isomorphism
\begin{equation}
\label{eqn:decomp-CA}
(\bU^-)^A \times_S \bT^A \times_S \bU^A \simto \bC^A.
\end{equation}
Moreover, by Lemma~\ref{lem:fixed-pts-properties}\eqref{it:fixed-pts-immersion} the morphism $\bC^A \to \bG^A$ induced by the open immersion $\bC \to \bG$ is itself an open immersion.

\begin{lem}
\phantomsection
\label{lem:flatness-CA}
\begin{enumerate}
\item
\label{it:flatness-CA-1}
The scheme $\bT^A$ is flat over $S$. It is smooth over $S$ iff the order of the torsion subgroup of $M_A$ is prime to all residual characteristics of $S$.
\item
\label{it:flatness-CA-2}
The scheme $\bU^A$ is flat over $S$. It is smooth over $S$ iff it satisfies the following condition: if $\Roots$ has an indecomposable component of type $\mathsf{A}_{2n}$ for some $n \geq 1$ whose stabilizer in $A$ acts nontrivially on this component, then $2$ is invertible on $S$.
\item
\label{it:flatness-CA-3}
The scheme $\bC^A$ is flat over $S$. It is smooth over $S$ iff the conditions in Theorem~\ref{thm:fixed-pts}\eqref{it:fixed-pts-3} hold.
\end{enumerate}
\end{lem}

\begin{proof}
\eqref{it:flatness-CA-1} This follows from the discussion in~\S\ref{ss:ex-diag-gps}.

\eqref{it:flatness-CA-2}
Consider the equivalence relation $\sim$ on the set of positive roots $\Roots_+$ that was defined in~\S\ref{ss:equiv-relation} (for the given action of $A$ on $\Roots$). In~\S\ref{ss:SL2-maps}, we have associated to each equivalence class $E$ for $\sim$ a subgroup scheme $\bU_E \subset \bG$.
After choosing a numbering $E_1, \ldots, E_n$ of these equivalence classes,
the product morphism induces an isomorphism
\[
\bU_{E_1} \times_S \cdots \times_S \bU_{E_n} \simto \bU.
\]
 Using Lemma~\ref{lem:fixed-pts-properties}\eqref{it:fixed-pts-immersion} we deduce that multiplication induces an isomorphism
\begin{equation}
\label{eqn:ua-product}
(\bU_{E_1})^A \times_S \cdots \times_S (\bU_{E_n})^A \simto \bU^A.
\end{equation}
Here, the factors are described in~\eqref{eqn:root-subgp-type-i} or~\eqref{eqn:root-subgp-type-ii}.

This shows that $\bU^A$ is a product of factors which are either isomorphic to $\bG_{\mathrm{a},S}$ or to $(\mathrm{U}_{3,S})^{\Z/2\Z}$; moreover, 
this second case occurs iff $\Roots$ has an indecomposable component of type $\mathsf{A}_{2n}$ for some $n \geq 1$ whose stabilizer in $A$ acts nontrivially on this factor. Each of these factors is flat over $S$ (see~\S\ref{sss:sl3_action} for the second case) so $\bU^A$ is flat. If $2$ is not a residual characteristic of $S$ or if 
no factor $(\mathrm{U}_{3,S})^{\Z/2\Z}$ occurs then all of these factors are smooth over $S$ (see again~\S\ref{sss:sl3_action}), so that $\bU^A$ is also smooth over $S$. On the other hand, if $2$ is a residual characteristic of $S$ and a factor $(\mathrm{U}_{3,S})^{\Z/2\Z}$ occurs, then as in~\S\ref{sss:sl3_action} one sees that there exists $s \in S$ such that $\Spec(\kappa(s)) \times_S \bU^A$ is not reduced, so that $\bU^A$ is not smooth over $S$.

\eqref{it:flatness-CA-3} The conclusion of~\eqref{it:flatness-CA-2} of course also applies to $(\bU^-)^A$. Using~\eqref{eqn:decomp-CA} we deduce that $\bC^A$ is flat over $S$. If the conditions of Theorem~\ref{thm:fixed-pts}\eqref{it:fixed-pts-3} hold, then by~\eqref{it:flatness-CA-1} and~\eqref{it:flatness-CA-2} each of $\bT^A$, $\bU^A$ and $(\bU^-)^A$ is smooth over $S$, so that $\bC^A$ is smooth over $S$ by~\eqref{eqn:decomp-CA}. Conversely, if one of these conditions fails, then as in~\eqref{it:flatness-CA-2} one sees that there exists $s \in S$ such that $\Spec(\kappa(s)) \times_S \bC^A$ is not reduced, so that $\bC^A$ is not smooth over $S$.
\end{proof}

\begin{ex}
\label{ex:torus-sc-adj}
If $\bG$ is semisimple and either simply connected or of adjoint type, then $M_A$ is free since $M$ admits a basis permuted by $A$. (In the first case one can take the basis of fundamental weights, and in the second case the basis of simple roots.) Hence in these cases $\bT^A$ is a torus.
\end{ex}

\subsection{The case of algebraically closed fields}
\label{ss:fixed-pts-alg-closed-fields}

In this subsection we assume that $S=\Spec (\bk)$ is the spectrum of an algebraically closed field $\bk$. 
Below we prove in particular in this case that the reduced neutral component $(\bG^{A})^\circ_{\rm{red}}$ is a split reductive group.
As explained in the introduction this case was already treated by Adler--Lansky and Haines,
but for the reader's convenience we give a self-contained argument.
We simultaneously establish a number of facts about the structure theory of this reductive group, listed in Proposition~\ref{prop:fix-pts-fields} below, which will be useful for the further study in Section~\ref{sec:complements}.

In this statement, we denote by $\bN$ the normalizer of $\bT$ in $\bG$, and by $\bW$ the Weyl group of $(\bG,\bT)$, i.e.~the quotient $\bN/\bT$. (As explained in~\cite[Exp.~XII, \S 2]{sga3.2}, $\bN$ is a smooth subgroup scheme of $\bG$, and $\bW$ is a finite constant group scheme, see also \cite[Proposition~5.1.6]{conrad}. For simplicity, we will not distinguish $\bW$ from $\bW(\bk)$.) Since $A$ stabilizes $\bT$, it also stabilizes $\bN$, and this action induces an action on $\bW$.

\begin{prop}
\label{prop:fix-pts-fields}
	The following assertions hold:
	\begin{enumerate}
		\item 
		\label{it:fields-1}
		$(\bG^A)^\circ_{\mathrm{red}}$ is a connected reductive group;
		\item 
		\label{it:fields-2}
		$(\bT^A)^\circ_{\mathrm{red}} \subset (\bG^A)^\circ_{\mathrm{red}}$ is a 
		maximal torus, and $(\bB^A)^\circ_{\mathrm{red}}$ is a Borel subgroup;
		\item 
		\label{it:fields-3}
		the positive roots of $((\bG^A)^\circ_{\mathrm{red}},(\bT^A)^\circ_{\mathrm{red}})$ determined by $(\bB^A)^\circ_{\mathrm{red}}$ are in natural bijection with the equivalence classes for the equivalence relation $\sim$ on $\Roots_+$ (see~\eqref{eqn:roots-equiv});
		\item 
		\label{it:fields-4}
		for any equivalence class $E$, seen as a positive root for $((\bG^A)^\circ_{\mathrm{red}},(\bT^A)^\circ_{\mathrm{red}})$, the corresponding root subgroup of $(\bG^A)^\circ_{\mathrm{red}}$ is $((\bU_E)^A)_{\mathrm{red}}$;
		\item 
		\label{it:fields-Bruhat-dec}
		the natural maps provide bijections
		\[
		\bN^A(\bk)/\bT^A(\bk) \cong \bW^A \cong \bB^A(\bk)\backslash \bG^A(\bk)/\bB^A(\bk);
		\] 
		\item 
		\label{it:fields-C-dense}
		the open subscheme $\bC^A \subset \bG^A$ is dense;
		\item
		\label{it:fields-conn-components}
		the embedding $\bT^A \to \bG^A$ induces an isomorphism between the groups of connected components of these group schemes;
		\item
		\label{it:fields-Weyl-gp}
		the Weyl group of $(\bG^A)^\circ_{\mathrm{red}}$ with respect to $(\bT^A)^\circ_{\mathrm{red}}$ identifies canonically with $\bW^A$.
	\end{enumerate}
\end{prop}

\begin{proof} 
Let $\mathbf{K}$ be the unipotent radical of $(\bG^A)_{\mathrm{red}}^\circ$, and set $\mathbf{R} := (\bG^A)_{\mathrm{red}}^\circ / \mathbf{K}$ which is a connected reductive group over $\bk$.
Let $q: (\bG^A)_{\mathrm{red}}^\circ \to \mathbf{R}$ be the quotient map.

\textit{Step 1. The restriction of $q$ to a map $(\bT^A)^\circ_{\mathrm{red}} \to \mathbf{R}$ is a closed immersion.  In particular, $\mathbf{R}$ contains a torus of dimension $\dim(\bT^A)$.}  
This follows from the observation that the subgroup  $(\bT^A)^\circ_{\mathrm{red}}=\Diag_{\bk}(M_A / (M_A)_{\mathrm{tor}})$ is a torus.
It therefore has trivial intersection with the unipotent group $\mathbf{K}$.

\textit{Step 2.  For each equivalence class $E$, there is a map
\[
\varphi_E: \mathrm{SL}_{2,\bk} \to \bG^A
\]
whose restriction to the subgroup $\mathrm{U}_{2,\bk} \subset \mathrm{SL}_{2,\bk}$ of upper triangular unipotent matrices induces an isomorphism
\[
\mathrm{U}_{2,\bk} \simto (\bU_E)^A_{\mathrm{red}}.
\]}
If $E$ is of type (i$'$), set $\varphi_E = \imath_E$.  We have seen in~\eqref{eqn:root-subgp-type-i} that this map behaves as desired on $\mathrm{U}_{2,\bk}$.  (In this case, the scheme $(\bU_E)^A$ is already reduced.)  If $E$ is of type (ii$'$), consider the sequences of maps
\begin{align*}
&\mathrm{U}_{2,\bk} \hookrightarrow \mathrm{SL}_{2,\bk} \twoheadrightarrow \mathrm{PGL}_{2,\bk} \xrightarrow{\eqref{eqn:SL3-fixed-pts}} (\mathrm{SL}_{3,\bk})^{\Z/2\Z} \xrightarrow{\imath_E} \bG^A &&\text{if $\mathrm{char}(\bk) \ne 2$;}  \\
&\mathrm{U}_{2,\bk} \hookrightarrow \mathrm{SL}_{2,\bk} \xrightarrow{\eqref{eqn:SL3-fixed-pts-2}} (\mathrm{SL}_{3,\bk})^{\Z/2\Z} \xrightarrow{\imath_E} \bG^A &&\text{if $\mathrm{char}(\bk) = 2$.} 
\end{align*}
We have seen in ~\eqref{eqn:root-subgp-type-ii} that $\imath_E$ restricts to an isomorphism $(\mathrm{U}_{3,\bk})^{\Z/2\Z} \simto (\bU_E)^A$, so the claim follows from the observation that either~\eqref{eqn:SL3-fixed-pts} or~\eqref{eqn:SL3-fixed-pts-2} (depending on the characteristic of $\bk$) induces an isomorphism $\mathrm{U}_{2,\bk} \simto (\mathrm{U}_{3,\bk})^{\Z/2\Z}_{\mathrm{red}}$.  
(Recall that the scheme $(\mathrm{U}_{3,\bk})^{\Z/2\Z}$ is reduced if and only if $\mathrm{char}(\bk) \ne 2$.)

\textit{Step 3.  The map $\gamma_E^\vee: \mathbb{G}_{\mathrm{m,\bk}} \to \bG^A$ given by
\[
\gamma_E^\vee(a) = \varphi_E \begin{pmatrix} a & 0 \\ 0 & a^{-1} \end{pmatrix}
\]
is a nontrivial cocharacter of $(\bT^A)^\circ_{\mathrm{red}}$.}
The fact that $\gamma_E^\vee$ factors through $(\bT^A)^\circ_{\mathrm{red}}$ follows from the fact that $\mathbb{G}_{\mathrm{m,\bk}}$ is reduced and connected.
To check that $\gamma_E^\vee$ is nonzero, we simply remark that in $M^\vee$,  by~\eqref{eqn:root-subgp-type-i-torus} and~\eqref{eqn:root-subgp-type-ii-torus} (combined with~\eqref{eqn:SL3-fixed-pts} and~\eqref{eqn:SL3-fixed-pts-2} in the latter case), we have
\[
\gamma_E^\vee = 
\begin{cases}
\sum_{\alpha \in E} \alpha^\vee &\text{if $E$ is of type (i$'$),} \\
2\sum_{\alpha \in E,\text{ $\alpha$ special}} \alpha^\vee &\text{if $E$ is of type (ii$'$) and $\mathrm{char}(\bk) \ne 2$,} \\
\sum_{\alpha \in E,\text{ $\alpha$ special}} \alpha^\vee &\text{if $E$ is of type (ii$'$) and $\mathrm{char}(\bk) = 2$.}
\end{cases}
\]
For future reference, we rewrite this formula as follows, using the observation that in type (ii$'$), the sum of all nonspecial roots is equal to the sum of all special roots:
\begin{equation}
\label{eqn:coroot-formula-pre}
\gamma_E^\vee = 
\begin{cases}
\sum_{\alpha \in E} \alpha^\vee &\text{if $E$ is of type (i$'$),} \\
& \qquad\text{or if $E$ is of type (ii$'$) and $\mathrm{char}(\bk) \ne 2$,} \\
\sum_{\alpha \in E,\text{ $\alpha$ special}} \alpha^\vee &\text{if $E$ is of type (ii$'$) and $\mathrm{char}(\bk) = 2$.}
\end{cases}
\end{equation}

\textit{Step 4. For each equivalence class $E$, let $\bar\alpha_E$ be the weight by which the torus $(\bT^A)^\circ_{\mathrm{red}}$ acts on $\Lie((\bU_E)^A_{\mathrm{red}})$.  Then each $\bar\alpha_E$ is nonzero.}
Since $\Lie((\bU_E)^A_{\mathrm{red}})$ is a subspace of $\Lie(\bU_E)$, it is enough to check that $(\bT^A)^\circ_{\mathrm{red}}$ has no nonzero fixed vectors in $\Lie(\bU_E)$.  Consider the cocharacter $\gamma_E^\vee$ from Step~3.  The claim follows from the fact that $\langle \alpha, \gamma_E^\vee \rangle > 0$ for all $\alpha \in E$.  More precisely, a straightforward computation (by reducing to the quasi-simple case) shows that
\begin{align*}
\langle \alpha, \gamma_E^\vee \rangle &= 2\quad\text{for all $\alpha \in E$} &&\text{if $E$ is of type (i$'$),} \\
\langle \alpha, \gamma_E^\vee \rangle &=
\begin{cases}
1 & \text{if $\alpha \in E$ is nonspecial and $\mathrm{char}(\bk) = 2$,} \\
2 & \text{if $\alpha \in E$ is nonspecial and $\mathrm{char}(\bk) \ne 2$,} \\
2 & \text{if $\alpha \in E$ is special and $\mathrm{char}(\bk) = 2$,} \\
4 & \text{if $\alpha \in E$ is special and $\mathrm{char}(\bk) \ne 2$,}
\end{cases}
&&\text{and $E$ is of type (ii$'$).}
\end{align*}

\textit{Step 5.  For each equivalence class $E$ for $\sim$, the morphism $(\bU_{E})^A_{\mathrm{red}} \to \mathbf{R}$ obtained by restricting $q$ is finite, with kernel the $r$-th infinitesimal neighborhood of the unit for some $r \geq 1$.} 
The kernel of the morphism under consideration is a subgroup scheme of $(\bU_E)^A_{\mathrm{red}} \cong \mathbb{G}_{\mathrm{a},\bk}$, which is stable under the action of $(\bT^A)^\circ_{\mathrm{red}}$ by conjugation; if it is not an infinitesimal neighborhood of the unit then it is
all of $(\bU_{E})^A_{\mathrm{red}}$.  
By Step~2, it follows that the composition
\[
\mathrm{U}_{2,\bk} \hookrightarrow \mathrm{SL}_{2,\bk} \xrightarrow{\varphi_{E}} (\bT^A)^\circ_{\mathrm{red}}\subset (\bG^A)^\circ_{\mathrm{red}} \to \mathbf{R}
\]
is trivial.  
The kernel of $\mathrm{SL}_{2,\bk} \to \mathbf{R}$ is therefore a normal subgroup scheme of $\mathrm{SL}_{2,\bk}$ that contains $\mathrm{U}_{2,\bk}$.  
The only such subgroup is all of $\mathrm{SL}_{2,\bk}$: so the map $\mathrm{SL}_{2,\bk} \to \mathbf{R}$ is trivial.  
But Steps~1 and~3 together imply that this map is nontrivial on $\{(\begin{smallmatrix} a & 0 \\ 0 & a^{-1} \end{smallmatrix})\} \subset \mathrm{SL}_{2,\bk}$, a contradiction.  

\textit{Step 6. If $E_1$ and $E_2$ are distinct equivalence classes, then $\bar\alpha_{E_1} $ and $\bar\alpha_{E_2}$ are linearly independent.}  Consider the morphism $\bG_\sconn \to \bG$ as in~\S\ref{ss:reduction-scqs}, and the maximal torus $\bT_\sconn$ of $\bG_\sconn$. As explained in Example~\ref{ex:torus-sc-adj}, $(\bT_\sconn)^A$ is reduced and connected. The morphism $\bG_\sconn \to \bG$ induces a morphism $(\bT_\sconn)^A \to (\bT^A)^\circ_{\mathrm{red}}$, hence a morphism relating characters of these tori. 
Explicitly, we have
\[
(\bT_\sconn)^A = \Diag_{\bk}((M_\sconn)_A), \quad (\bT^A)^\circ_{\mathrm{red}} = \Diag_{\bk}(M_A / (M_A)_{\mathrm{tor}})
\]
where $M_\sconn$ is as in~\S\ref{ss:reduction-scqs}, i.e.~$M_\sconn$ is the lattice of weights of the root system $\Roots$, and this morphism is given by the obvious morphism
\[
M_A / (M_A)_{\mathrm{tor}} \to (M_\sconn)_A.
\]
The desired claim therefore follows from Lemma~\ref{lem:roots-coinvariants}, using the embedding $(M_\sconn)_A \to \mathbb{R} \otimes_{\Z} (M_\sconn)_A = (\mathbb{R} \otimes_{\Z} M_\sconn)_A$.

\textit{Step 7.  Proof of parts~\eqref{it:fields-1}, \eqref{it:fields-2}, \eqref{it:fields-3}, and~\eqref{it:fields-4}.}  
All four of these assertions will follow from the fact that $q: (\bG^A)_{\mathrm{red}}^\circ \to \mathbf{R}$ is an isomorphism.
To prove this fact, it is enough to show the equality
\[
\dim((\bG^A)_{\mathrm{red}}^\circ) = \dim(\mathbf{R}).
\]
It is obvious that $\dim((\bG^A)_{\mathrm{red}}^\circ) \ge \dim(\mathbf{R})$; we need only prove the opposite inequality.  Let $N$ be the number of equivalence classes for $\sim$.  The analysis of $\bC^A$ in~\S\ref{ss:big-cell} shows that $\dim ((\bG^A)_{\mathrm{red}}^\circ) = \dim ((\bT^A)_{\mathrm{red}}^\circ) + 2N$.  On the other hand, using Steps~1, 5, and~6 and considering the action of $(\bT^A)_{\mathrm{red}}^\circ$ on the Lie algebra of $\mathbf{R}$ we see that $\dim (\mathbf{R}) \ge \dim ((\bT^A)_{\mathrm{red}}^\circ) + 2N$, so we are done. 

\textit{Step 8. The natural morphism
\begin{equation}
\label{eqn:NA-WA}
\bN^A(\bk)\to \bW^A
\end{equation}
is surjective.}
By~\cite[Corollaire~3.5]{hee}, $\bW^A$ admits a system of Coxeter generators in bijection with orbits of $A$ in $\Delta$. For any such orbit the construction of~\S\ref{ss:SL2-maps} provides a twisted $\mathrm{SL}_2$-map with domain either $\mathrm{SL}_{2,\bk}$ or $(\mathrm{SL}_{3,\bk})^{\Z/2\Z}$. In the first, resp.~second, case, the image of the matrix $\left( \begin{smallmatrix} 0 & 1 \\ -1 & 0 \end{smallmatrix} \right)$, resp.~$\left( \begin{smallmatrix} 0 & 0 & 1 \\ 0 & -1 & 0 \\ 1 & 0 & 0 \end{smallmatrix} \right)$, provides a representative of the corresponding reflection in $\bW^A$.

\textit{Step 9. Study of the Bruhat decomposition.}  
Recall the Bruhat decomposition
\[
\bG = \bigsqcup_{w \in \bW} \bU w \bB = \bigsqcup_{w \in \bW} \bB w \bB,
\]
where each $\bU w \bB = \bB w\bB$ is a locally closed subscheme of $\bG$.
For any $a \in A$, the action of $a$ induces an isomorphism
\[
\bU w \bB \simto \bU (a \cdot w) \bB = \bB (a \cdot w) \bB;
\]
it follows that $\bG^A \cap (\bU w \bB) = \bG^A \cap (\bB w \bB)$ is empty unless $w \in \bW^A$, so that we have a decomposition $\bG^A = \bigsqcup_{w \in \bW^A} (\bU w \bB)^A$.  By surjectivity of~\eqref{eqn:NA-WA} each $w \in \bW^A$ admits a lift $\dot{w} \in \bN^A(\bk)$. Using such a lift and the usual description of Bruhat cells (see e.g.~\cite[Equation~(1) in~\S II.13.2]{jantzen}) we deduce that for any $w \in \bW^A$ we have
\[
(\bU w \bB)^A = \bU^A \dot{w} \bB^A = \bB^A \dot{w} \bB^A.
\]
We conclude that there is a decomposition
\begin{equation}\label{eqn:fields-Bruhat-decomp}
\bG^A = \bigsqcup_{w \in \bW^A} \bU^A \dot{w} \bB^A = \bigsqcup_{w \in \bW^A} \bB^A \dot{w} \bB^A.
\end{equation}

Let us now note that for $w \in \bW^A$ and $E$ an equivalence class for $\sim$, we either have $w(E) \subset \Roots_+$ or $w(E) \subset -\Roots_+$.  In the latter case, we write $w(E) < 0$.  Using~\eqref{eqn:ua-product}, we obtain
\begin{equation}
\label{eqn:fields-Bruhat-affine}
(\bU w \bB)^A = \left(\prod_{\substack{\text{$E$ equiv. class for $\sim$} \\ w^{-1}(E) < 0}} (\bU_E)^A\right) \dot {w} \bB^A.
\end{equation}
It then follows that
\[
\dim(\bU^A \dot{w} \bB^A) = \# \{E \in (\Roots_+ / \sim) \, \mid w^{-1}(E)<0\} + \dim (\bB^A).
\]

\textit{Step 10. Proof of~\eqref{it:fields-Bruhat-dec}.}  
The first bijection follows from Step~8 and the observation that the kernel of~\eqref{eqn:NA-WA} is $\bT(\bk) \cap \bN^A(\bk) = \bT^A(\bk)$; the second bijection is a restatement of~\eqref{eqn:fields-Bruhat-decomp}.

\textit{Step 11. Proof of~\eqref{it:fields-C-dense}.}
Let $w_0 \in \bW$ be the longest element.  Since the action of $A$ preserves lengths, we have $w_0 \in \bW^A$.  The subscheme $\bB^A \dot{w}_0 \bB^A$ is the unique term of maximal dimension in~\eqref{eqn:fields-Bruhat-decomp}.  Since all irreducible components of $\bG^A$ have the same dimension (because it is a group scheme), we deduce that $\bB^A \dot{w}_0 \bB^A$ is dense in $\bG^A$. Finally, comparing~\eqref{eqn:fields-Bruhat-affine} (for $\dot{w}_0$) with~\eqref{eqn:decomp-CA}, we obtain an identification $\bB^A \dot{w}_0 \bB^A \cong \bC^A$.

\textit{Step 12. Proof of~\eqref{it:fields-conn-components}.}
By Step~11, the embedding $\bC^A \to \bG^A$ induces a bijection between sets of connected components. In view of~\eqref{eqn:decomp-CA}, and since $(\bU^A)_{\mathrm{red}}$ and $((\bU^-)^A)_{\mathrm{red}}$ are isomorphic to affine spaces, the same property holds for the embedding
\[
\bT^A \to \bC^A,
\] 
which proves~\eqref{it:fields-conn-components}.

\textit{Step 13. Proof of~\eqref{it:fields-Weyl-gp}.}
Thanks to Step~12, Example~\ref{ex:torus-sc-adj} shows that $\bG^A$ is connected if $\bG$ is semisimple and simply connected. The study of the morphism~\eqref{eqn:NA-WA} above therefore shows that this surjection admits a (set theoretic) section which takes values in $\bN^A(\bk) \cap (\bG^A)^\circ(\bk)$. 
From now on, we assume that the lifts $\dot{w}$ (for $w \in \bW^A$) are chosen in $(\bG^A)^\circ_{\textrm{red}}(\bk)=(\bG^A)^\circ(\bk)$. 
Let us also fix some representatives $t_1, \dots, t_r \in \bT^A(\bk)$ for the connected components of $\bT^A$. 
Then we have
\[
\bN^A = \bigsqcup_{\substack{w \in \bW^A \\ i \in \{1, \dots, r\}}} \dot{w} t_i (\bT^A)^\circ,
\]
and hence
\[
(\bN^A) \cap (\bG^A)^\circ_{\mathrm{red}} = \bigsqcup_{w \in \bW^A} \dot{w} (\bT^A)^\circ_{\mathrm{red}}.
\]
On the other hand, there exists a natural closed immersion
\[
(\bN^A) \cap (\bG^A)^\circ_{\mathrm{red}} \to N_{(\bG^A)^\circ_{\mathrm{red}}}((\bT^A)^\circ_{\mathrm{red}}),
\]
and we deduce an embedding
\begin{equation}
\label{eqn:embedding-Weyl-gp}
\bW^A \hookrightarrow N_{(\bG^A)^\circ_{\mathrm{red}}}((\bT^A)^\circ_{\mathrm{red}}) / (\bT^A)^\circ_{\mathrm{red}}.
\end{equation}
For any $w \in \bW^A$ we similarly have
\[
(\bB^A \dot{w} \bB^A) \cap (\bG^A)^\circ_{\mathrm{red}} = (\bB^A)^\circ_{\mathrm{red}} \dot{w} (\bB^A)^\circ_{\mathrm{red}},
\]
hence
\[
(\bG^A)^\circ_{\mathrm{red}} = \bigsqcup_{w \in \bW^A} (\bB^A)^\circ_{\mathrm{red}} \dot{w} (\bB^A)^\circ_{\mathrm{red}},
\]
and finally a bijection
\[
\bW^A \simto (\bB^A)^\circ_{\mathrm{red}}(\bk) \backslash (\bG^A)^\circ_{\mathrm{red}}(\bk) / (\bB^A)^\circ_{\mathrm{red}}(\bk).
\]
Using the Bruhat decomposition for the reductive group $(\bG^A)^\circ_{\mathrm{red}}$, we deduce that the morphism~\eqref{eqn:embedding-Weyl-gp} is an isomorphism, which finally proves~\eqref{it:fields-Weyl-gp}.
\end{proof}

\begin{rmk}
\label{rmk:field-root-calc}
Proposition~\ref{prop:fix-pts-fields}\eqref{it:fields-3} says that the positive roots for $(\bG^A)^\circ_{\mathrm{red}}$ are in bijection with the set of equivalence classes for $\sim$, and in Step~5 of the proof, we introduced the notation $\bar\alpha_E$ for the root corresponding to an equivalence class $E$.  This is a character of $(\bT^A)^\circ_{\mathrm{red}}$, or, equivalently, an element of $M_A/(M_A)_{\mathrm{tor}}$.  It is immediate from the definition of $\bU_E$ that $\bar\alpha_E$ is the image of some element $\alpha \in E$ under the projection $M \to M_A/(M_A)_{\mathrm{tor}}$ (and this observation was adequate for Step~6 of the proof).

But of which element is $\bar\alpha_E$ the image?  If $E$ is of type (i$'$), then it consists of a single $A$-orbit, and its image in $M_A/(M_A)_{\mathrm{tor}}$ is a singleton.  But if $E$ is of type (ii$'$), its image in $M_A/(M_A)_{\mathrm{tor}}$  consists of two elements: one is the image of the nonspecial roots, and the other is the image of the special roots.  The latter is twice the former.  We claim that
$\bar\alpha_E$ is the image in $M_A/(M_A)_{\mathrm{tor}}$ of
\[
\begin{cases}
\text{any root in $E$} & \text{if $E$ is of type (i$'$),} \\
\text{any nonspecial root in $E$} & \text{if $E$ is of type (ii$'$) and $\mathrm{char}(\bk) \ne 2$,} \\
\text{any special root in $E$} & \text{if $E$ is of type (ii$'$) and $\mathrm{char}(\bk) = 2$.}
\end{cases}
\]
To justify the latter two cases, we use~\eqref{eqn:root-subgp-type-ii} to reduce the problem to the setting of $\bG = \mathrm{SL}_{3,\bk}$.  In this setting, the claim follows from~\eqref{eqn:u3-fixed-field}.

With the formula for $\bar\alpha_E$ in hand, we read off from Step~5 of the proof that
\[
\langle \bar\alpha_E, \gamma_E^\vee \rangle = 2
\]
for any $E$.  Since $\gamma_E^\vee$ comes from a homomorphism $\mathrm{SL}_{2,\bk} \to \bG^A$ as in Step~2, we conclude that $\gamma_E^\vee$ is in fact the coroot for $(\bG^A)^\circ_{\mathrm{red}}$ corresponding to $\bar\alpha_E$.
\end{rmk}

\begin{rmk}
Let us explain an argument proving
the surjectivity of~\eqref{eqn:NA-WA} which does not rely on the results of~\cite{hee}.
We denote by $\Roots^{(A)}$, resp.~$\Roots^{(A)}_+$, the image of $\Roots$, resp.~$\Roots_+$, in $M_A$. This subset will be studied more thoroughly in~\S\ref{ss:root-data} below; for now, we note that Steps~5 and~6 in the preceding proof show that the natural map
$\Roots \to \Roots^{(A)}$ induces a bijection $\Roots/A \simto \Roots^{(A)}$ and that,
by parts~\eqref{it:fields-3}--\eqref{it:fields-4} of the proposition, the root system $\Roots^{(A) \prime}$ of $(\bG^A)_{\mathrm{red}}^\circ$ is obtained from $\Roots^{(A)}$ by discarding one element from each pair of the form $\{\alpha,2\alpha\}$. In particular, $\Roots^{(A)}_+$ defines a positive system $\Roots^{(A) \prime}_+$ in $\Roots^{(A) \prime}$.
Since $A$-invariant elements of $\bW$ act compatibly on $\Roots$ and on $M_A$, $\bW^A$ acts on $\Roots^{(A)}$.
If $w \in \bW^A$ then $w(\Roots^{(A) \prime}_+)$ is a positive system for $\Roots^{(A) \prime}$; hence there exists $n \in N_{(\bG^A)_{\mathrm{red}}^\circ}((\bT^A)_{\mathrm{red}}^\circ)(\bk)$ such that $w(\Roots^{(A) \prime}_+)=n(\Roots^{(A) \prime}_+)$. The argument in the proof of Lemma~\ref{lem:centralizer.norm.torus.weyl.gp}\eqref{it:Weyl-group-2} below shows that $N_{(\bG^A)_{\mathrm{red}}^\circ}((\bT^A)_{\mathrm{red}}^\circ)(\bk) \subset \bN^A(\bk)$. In particular $n$ defines an element $w'$ in $\bW^A$ such that $w^{-1} w'$ induces a based automorphism of $\Roots$. Hence $w'=w$, which proves the desired claim.
\end{rmk}

\subsection{Proof of Theorem~\ref{thm:fixed-pts}}

We can finally explain the proof of Theorem~\ref{thm:fixed-pts}.
For part~\eqref{it:fixed-pts-1}, we will study the multiplication morphism
\[
 f: \bC^A \times_S \bC^A \to \bG^A,
\]
along with its base changes to points or geometric points of $S$.

First, let $\overline{s}: \Spec (\bk) \to S$ be a geometric point. 
By Proposition~\ref{prop:fix-pts-fields}\eqref{it:fields-C-dense} the open subscheme 
$\bC_{\overline{s}}^A \subset \bG_{\overline{s}}^A$ is dense, and hence $f_{\overline{s}}$ is surjective.  Since any point of a scheme is the image of a geometric point, it follows that $f$ is surjective.

Next, let $s \in S$, and consider the base change $f_s : \bC^A_s \times_{\kappa(s)} \bC^A_s \to \bG^A_s$.  This map factors as the composition
\begin{multline*}
\bC^A_s \times_{\kappa(s)} \bC^A_s \xrightarrow{\text{inclusion}}
\bG^A_s \times_{\kappa(s)} \bG^A_s \xrightarrow[\sim]{(g,h) \mapsto (g,gh)}
\bG^A_s \times_{\kappa(s)} \bG^A_s \\ \xrightarrow{\text{projection to 2nd factor}}
\bG^A_s.
\end{multline*}
The first map is an open immersion (by Lemma~\ref{lem:fixed-pts-properties}\eqref{it:fixed-pts-immersion}); the second is an isomorphism; and the third is flat, since $\bG^A_s$ is (obviously) flat over $\Spec(\kappa(s))$.  We conclude that $f_s$ is flat.

The two preceding paragraphs and Lemma~\ref{lem:flatness-CA}\eqref{it:flatness-CA-3} show that $f$ satisfies the assumptions of the fiberwise criterion for flatness (see~\cite[\href{https://stacks.math.columbia.edu/tag/039E}{Tag 039E}]{stacks-project}). This criterion implies that $\bG^A$ is flat over $S$, i.e.~that~\eqref{it:fixed-pts-1} holds.

Part~\eqref{it:fixed-pts-2} follows from Proposition~\ref{prop:fix-pts-fields}\eqref{it:fields-conn-components} and
the comments following Lemma~\ref{lem:torus_fix_pts_diag}.

Regarding part~\eqref{it:fixed-pts-3}, if $\bG^A$ is smooth then so is the open subscheme $\bC^A$, so that the conditions in the statement must be satisfied by Lemma~\ref{lem:flatness-CA}\eqref{it:flatness-CA-3}. Conversely, if these conditions are satisfied then $\bC^A$ is smooth. We deduce that for any $s \in S$ the fiber $\bG_s^A$ admits a smooth open subscheme containing the unit, and is therefore smooth by~\cite[II, \S5, Th\'eor\`eme~2.1]{dg}. Since $\bG^A$ is known to be flat over $S$, by~\cite[\href{https://stacks.math.columbia.edu/tag/01V8}{Tag 01V8}]{stacks-project} this implies that $\bG^A$ is smooth, and finishes the proof.

 Finally, regarding~\eqref{it:fixed-pts-5}, we have seen that $\bG^A$ is flat over $S$ in~\eqref{it:fixed-pts-1}. It is clearly affine, and its generic fiber is smooth by~\eqref{it:fixed-pts-3} and geometrically connected (and hence connected) by~\eqref{it:fixed-pts-2}. The identity connected component of its geometric special fiber is reductive by Proposition~\ref{prop:fix-pts-fields}\eqref{it:fields-1}, which finishes the verification of the conditions of the definition in~\cite{py}. (These conditions include an extra condition on comparison of dimensions, but this condition is automatic when the group scheme is of finite type, which is the case here, as explained in the discussion following~\cite[Theorem~1.2]{py}.)
 
 \begin{rmk}
 \label{rmk:f-epi}
 Consider the morphism $f$ used in the proof of Theorem~\ref{thm:fixed-pts}\eqref{it:fixed-pts-1}. The fiberwise criterion for flatness also implies that $f$ is flat. Since it is also surjective, it is faithfully flat. It is of finite type, hence an fppf cover, and in particular an epimorphism.
 \end{rmk}
 
 \begin{cor}
 \label{cor:connected-fixed-pts-reductive}
 Assume $\bG$ and $A$ satisfy the conditions of Theorem~\ref{thm:fixed-pts}\eqref{it:fixed-pts-3}.  Let $(\bG^A)^\circ \subset \bG^A$ be the ``fiberwise identity component,'' i.e., the open subgroup scheme characterized by the following property:
 \begin{quote}
 For each point $s \in S$, $(\bG^A)_s^\circ$ is the connected component of the unit in the smooth group scheme $(\bG^A)_s$ over $\kappa(s)$.
 \end{quote}
 Then $(\bG^A)^\circ$ is a split reductive group scheme over $S$.  In particular, if $\bG$ is semisimple and either simply connected or of adjoint type, then $\bG^A$ is a split reductive group scheme over $S$.
 \end{cor}
For the existence of the subgroup scheme $(\bG^A)^\circ$, see~\cite[Exp.~$\text{VI}_\text{B}$, \S 3]{sga31}.
\begin{proof}
Let $N$ be the order of the torsion subgroup of $M_A$, multiplied by $2$ if $\Roots$ has an indecomposable constituent of type $\mathsf{A}_{2n}$ with a nontrivial action of its stabilizer in $A$. The conditions in Theorem~\ref{thm:fixed-pts}\eqref{it:fixed-pts-3} imply that $S$ admits a map to $\Spec(\Z[\frac{1}{N}])$.  Then $\bG$ (together with the action of $A$) can be obtained by base change from a pinned reductive group scheme over $\Spec(\Z[\frac{1}{N}])$; by compatibility of the operation $(-)^\circ$ with base change (see~\cite[Exp.~$\text{VI}_\text{B}$, Proposition~3.3]{sga31}), this reduces the proof to the case $S=\Spec(\Z[\frac{1}{N}])$. Now $(\bG^A)^\circ$ is smooth, as an open subscheme of the smooth scheme $\bG^A$.
The open immersion $(\bG^A)^\circ \to \bG^A$ is quasi-compact because $\bG^A$ is noetherian (see~\cite[\href{https://stacks.math.columbia.edu/tag/01OX}{Tag 01OX}]{stacks-project}), so $(\bG^A)^\circ$ is quasi-affine over $\Spec(\Z[\frac{1}{N}])$ by~\cite[\href{https://stacks.math.columbia.edu/tag/02JR}{Tag 02JR}]{stacks-project} and~\cite[\href{https://stacks.math.columbia.edu/tag/01SN}{Tag 01SN}]{stacks-project}. In view of~\cite[Exp.~$\text{VI}_\text{B}$, Proposition~12.9]{sga31}, it follows that $(\bG^A)^\circ$ is affine. Finally, each geometric fiber of $(\bG^A)^\circ$ is a connected reductive group, so $(\bG^A)^\circ$ is a reductive group scheme. The subgroup $(\bT^A)^\circ$ is a split maximal torus of $(\bG^A)^\circ$ by Proposition~\ref{prop:fix-pts-fields}\eqref{it:fields-2}.

For the last assertion, if $\bG$ is semisimple and either simply connected or of adjoint type, Example~\ref{ex:torus-sc-adj} and Proposition~\ref{prop:fix-pts-fields}\eqref{it:fields-conn-components} imply that every geometric fiber of $\bG^A$ is connected, so $(\bG^A)^\circ = \bG^A$.
\end{proof}
  
 \begin{ex}
 \label{ex:sl_odd}
 Let us come back to the example considered in~\S\ref{sss:action-SLodd}, with the pinning of Example~\ref{ex:roots-A2n}. (With this choice of pinning, it is easily seen that $\Z/2\Z$ acts by pinned automorphisms.)
 \begin{enumerate}
 \item
 \label{ex:sl_odd.1}
 First, assume that~$n=1$, i.e.~$\bG=\mathrm{SL}_{3,\Z}$. By Corollary~\ref{cor:connected-fixed-pts-reductive},
 \[
 \Spec(\Z[\tfrac{1}{2}]) \times_{\Spec(\Z)} (\mathrm{SL}_{3,\Z})^{\Z/2\Z}
 \]
 is a reductive group scheme over $\Spec(\Z[\frac{1}{2}])$. 
 In fact, considering the root data (see~\S\ref{ss:root-data} below) one sees that~\eqref{eqn:SL3-fixed-pts} is an isomorphism. 
 On the other hand, $\Spec(\F_2) \times_{\Spec(\Z)} (\mathrm{SL}_{3,\Z})^{\Z/2\Z}$ is not reduced, and~\eqref{eqn:SL3-fixed-pts-2} is an isomorphism.
 \item
  \label{ex:sl_odd.2}
Now, consider the case $n \in \Z_{\geq 2}$. Then $\bG^A$ is a flat and geometrically connected group scheme over $\Z$, the restriction
 \[
 \Spec(\Z[\tfrac{1}{2}]) \times_{\Spec(\Z)} (\mathrm{SL}_{2n+1,\Z})^{\Z/2\Z}
 \]
is isomorphic to $\mathrm{SO}_{2n+1,\Z[\frac{1}{2}]}$, but $ \Spec(\F_2) \times_{\Spec(\Z)} (\mathrm{SL}_{2n+1,\Z})^{\Z/2\Z}$ is nonreduced. The associated reduced group scheme is simple and simply connected of type $\mathsf{C}_{n}$, i.e.~isomorphic to $\mathrm{Sp}_{2n,\F_2}$. 
In particular, the base change of $\bG$ to $\Z_2$ is a quasi-reductive $\Z_2$-group scheme in the sense of Prasad--Yu~\cite{py} which is nonreductive.
 \end{enumerate}
 \end{ex}
 
\subsection{The case of general fields}
\label{ss:fixed-pts-fields}

For completeness, we explain in this subsection how to generalize the results of~\S\ref{ss:fixed-pts-alg-closed-fields} to more general base fields.
We therefore assume that $S=\Spec(\bk)$ for some field $\bk$, which we do \emph{not} assume to be algebraically closed. We choose an algebraic closure of $\bk$, which we denote by $\overline{\bk}$. For any scheme $X$ over $\bk$, we set $X_{\overline{\bk}} := \overline{\bk} \otimes_\bk X$. With this notation, by~\eqref{eqn:fixed-pts-bc} we have
\[
(\bG^A)_{\overline{\bk}} = (\bG_{\overline{\bk}})^A;
\]
this group scheme will be denoted $\bG^A_{\overline{\bk}}$.

\begin{prop}
\phantomsection
\label{prop:fixed-pts-fields-2}
\begin{enumerate}
\item
\label{it:fix-pts-fields-red}
The reduced subscheme $(\bG^A)_{\mathrm{red}}$ is geometrically reduced; as a consequence it is a subgroup scheme of $\bG^A$, and we have
$(\bG^A)_{\mathrm{red}, \overline{\bk}} = (\bG_{\overline{\bk}}^A)_{\mathrm{red}}$.
\item
\label{it:fields-reductive}
We have $((\bG^A)^\circ_{\mathrm{red}})_{\overline{\bk}} = (\bG_{\overline{\bk}}^A)^\circ_{\mathrm{red}}$, and $(\bG^A)^\circ_{\mathrm{red}}$ is a split reductive group over $\bk$.
\end{enumerate}
\end{prop}

\begin{proof}
\eqref{it:fix-pts-fields-red}
Consider once again the open subscheme $\bC^A \subset \bG^A$, and the decompositions~\eqref{eqn:decomp-CA} and~\eqref{eqn:ua-product}.
For any equivalence class $E$ in $\Roots_+$ we have the subgroup $((\bU_E)^A)_{\mathrm{red}}$ which is isomorphic to $\mathbb{G}_{\mathrm{a},\bk}$ (as in Step~2 of the proof of Proposition~\ref{prop:fix-pts-fields}); in particular $(\bU^A)_{\mathrm{red}}$ is an affine space over $\bk$. Similarly $((\bU^-)^A)_{\mathrm{red}}$ is an affine space over $\bk$. We deduce that $(\bC^A)_{\mathrm{red}}$ is the product of $(\bT^A)_{\mathrm{red}}$ and an affine space; in particular it is geometrically reduced.

Recall now the morphism
\[
f : \bC^A \times \bC^A \to \bG^A
\]
from the proof of Theorem~\ref{thm:fixed-pts}. 
Since $f$ is faithfully flat it induces an injective morphism
$\scO(\bG^A) \hookrightarrow \scO(\bC^A \times \bC^A)$,
hence an embedding
\[
\scO((\bG^A)_{\mathrm{red}}) \hookrightarrow \scO((\bC^A \times \bC^A)_{\mathrm{red}}).
\]
Here, since $(\bC^A)_{\mathrm{red}}$ is geometrically reduced we have $(\bC^A \times \bC^A)_{\mathrm{red}} = (\bC^A)_{\mathrm{red}} \times (\bC^A)_{\mathrm{red}}$, and this scheme is also geometrically reduced.
Hence $(\bG^A)_{\mathrm{red}}$ is geometrically reduced. The other statements in~\eqref{it:fix-pts-fields-red} are immediate consequences.

\eqref{it:fields-reductive}
Since the formation of the neutral component commutes with field extensions (as follows from~\cite[\href{https://stacks.math.columbia.edu/tag/04KV}{Tag 04KV}]{stacks-project}) we have
\[
((\bG^A)^\circ_{\mathrm{red}})_{\overline{\bk}} = ((\bG^A)_{\mathrm{red}, \overline{\bk}})^\circ,
\]
and by~\eqref{it:fix-pts-fields-red} the right-hand side coincides with $(\bG_{\overline{\bk}}^A)^\circ_{\mathrm{red}}$, which proves the desired equality. Since the latter group is connected and reductive by Proposition~\ref{prop:fix-pts-fields}\eqref{it:fields-1}, we deduce that $(\bG^A)^\circ_{\mathrm{red}}$ is a reductive group over $\bk$. The maximal torus $(\bT^A)^\circ_{\mathrm{red}}=\Diag_{\bk}(M_A / (M_A)_{\mathrm{tor}})$ is split, hence $(\bG^A)^\circ_{\mathrm{red}}$ is split.
\end{proof}

Once Proposition~\ref{prop:fixed-pts-fields-2} is established, the other structural properties of Proposition~\ref{prop:fix-pts-fields} follow for general base fields. In particular, here again the connected components of $\bG^A$ are in bijection with those of $\bT^A$; some of these connected components might therefore not be geometrically connected.
 
 \section{Complements}
 \label{sec:complements}
 
\subsection{Root data for fixed points}
\label{ss:root-data}

In this subsection we discuss the notion of roots and coroots for the fixed point group schemes $\bG^A$. We will give a description that only depends on the following data: a root datum
\[
\Phi = (M,\Roots, M^\vee, \Roots^\vee),
\]
a basis $\Delta \subset \Roots$, and an action of a group $A$ on $\Phi$ preserving $\Delta$. In this setting we will denote by $W_\Phi$ the Weyl group of $\Phi$, and by $\Roots^{(A)}$, resp.~$\Roots_+^{(A)}$, resp.~$\Delta^{(A)}$, the image of $\Roots$, resp.~$\Roots_+$, resp.~$\Delta$, along the quotient map
\[
M \to M_A.
\] 
Considering a pinned reductive group scheme with root datum $(M,\Roots, M^\vee, \Roots^\vee)$ and basis $\Delta$ over an algebraically closed field, Steps~5--6 in the proof of Proposition~\ref{prop:fix-pts-fields} (and the fact that $\mathbf{K}$ is trivial) show that the natural map
\[
\Roots \to \Roots^{(A)}
\]
induces a bijection $\Roots/A \simto \Roots^{(A)}$. In particular, we have
\[
\Roots^{(A)} = \Roots^{(A)}_+ \sqcup -\Roots^{(A)}_+.
\]

Recall the equivalence relation $\sim$ considered in~\S\ref{ss:equiv-relation}.
If an equivalence class $E$ is of type (ii$'$), write it as $E= E' \sqcup E''$ with $E'$ being the set of nonspecial roots, and $E''$ the set of special roots.  (Both $E'$ and $E''$ are $A$-orbits.) Now, we denote by:
\begin{itemize}
\item
$\Roots^{(A)}_{\mathrm{nd},\mathrm{nm},+} \subset \Roots_+^{(A)}$ the subset consisting of restrictions of equivalence classes $E$ of type (i$'$);
\item
$\Roots^{(A)}_{\mathrm{m},+} \subset \Roots_+^{(A)}$ the subset consisting of restrictions of $A$-orbits of the form $E'$ where $E=E'\sqcup E''$ is an equivalence class of type (ii$'$);
\item
$\Roots^{(A)}_{\mathrm{d},+} \subset \Roots_+^{(A)}$ the subset consisting of restrictions of $A$-orbits of the form $E''$ where $E=E'\sqcup E''$ is an equivalence class of type (ii$'$).
\end{itemize}
We also set
\[
\Roots^{(A)}_{\mathrm{nd},\mathrm{nm}} = \Roots^{(A)}_{\mathrm{nd},\mathrm{nm},+} \sqcup -\Roots^{(A)}_{\mathrm{nd},\mathrm{nm},+}, \ \ \Roots^{(A)}_{\mathrm{m}} = \Roots^{(A)}_{\mathrm{m},+} \sqcup -\Roots^{(A)}_{\mathrm{m},+}, \ \ \Roots^{(A)}_{\mathrm{d}} = \Roots^{(A)}_{\mathrm{d},+} \sqcup -\Roots^{(A)}_{\mathrm{d},+}.
\]
Then we have a partition
\[
\Roots^{(A)} = \Roots_{\mathrm{nd,nm}}^{(A)} \sqcup \Roots_{\mathrm{m}}^{(A)} \sqcup \Roots_{\mathrm{d}}^{(A)}, 
\]
and the assignment $\gamma \mapsto 2\gamma$ induces a bijection $\Roots_{\mathrm{m}}^{(A)} \simto \Roots_{\mathrm{d}}^{(A)}$. (Here, ``$\mathrm{m}$'' stands for ``multipliable,'' ``$\mathrm{d}$'' for ``divisible,'' and ``$\mathrm{n}$'' for ``non''.) We also set
\[
\Roots_1^{(A)} = \Roots_{\mathrm{nd,nm}}^{(A)} \sqcup \Roots_{\mathrm{m}}^{(A)}, \quad
\Roots^{(A)}_2 = \Roots_{\mathrm{nd,nm}}^{(A)} \sqcup \Roots_{\mathrm{d}}^{(A)}.
\]
Finally, we set
\[
\Delta^{(A)}_2 = (\Delta^{(A)} \cap \Roots^{(A)}_{\mathrm{nd},\mathrm{nm}}) \sqcup \{2\alpha : \alpha \in \Delta^{(A)} \cap \Roots^{(A)}_{\mathrm{m}} \}.
\]

It turns out that $\Roots^{(A)}$ ``extends'' to a root datum. For this, one has to define the set of coroots $\Roots^{(A),\vee}$ via the following formula. Let $\gamma \in \Roots^{(A)}_+$, and let $E$ be the unique equivalence class containing the $A$-orbit in $\Roots_+$ corresponding to $\gamma$. Then we set
\begin{equation}
\label{eqn:coroot-formula}
\gamma^\vee=
\begin{cases}
\sum_{\alpha \in E}\alpha^\vee & \text{if $\gamma \in \Roots^{(A)}_{\mathrm{nd}, \mathrm{nm}} \sqcup \Roots^{(A)}_{\mathrm{m}}$;} \\ 
  \sum_{\alpha \in E''} \alpha^\vee & \text{if $\gamma \in \Roots^{(A)}_{\mathrm{d}}$.}
\end{cases}
\end{equation}
If $\gamma \in -\Roots^{(A)}_+$, we set $\gamma^\vee := -(-\gamma)^\vee$. In both cases, we regard these coroots as elements in $(M^\vee)^A$.

If $\bT$ is a torus with character lattice $M$ (over a connected scheme $S$), and if we consider the action on $A$ induced by our given action on $M$,
the set of cocharacters $\bG_{\mathrm{m}} \to \bT^A$ identifies naturally with $(M^\vee)^A$, where $M^\vee$ is identified with the lattice of cocharacters $\bG_{\mathrm{m}} \to \bT$.
 Hence, the formulas above determine cocharacters of $\bT^A$. This contrasts with the character lattice $M_A$ of $\bT^A$ which may admit a non-trivial torsion subgroup $(M_A)_{\mathrm{tor}}$. The corresponding quotient $M_A / (M_A)_{\mathrm{tor}}$ is the weight lattice 
 of the maximal subtorus scheme 
 of $\bT^A$. 
Steps 5--6 in the proof of Proposition~\ref{prop:fix-pts-fields} show that the quotient morphism $M_A \to M_A / (M_A)_{\mathrm{tor}}$ is injective on $\Roots^{(A)}$, so that this set can (and will) also be regarded as a subset in $M_A / (M_A)_{\mathrm{tor}}$.
 We will denote by $\Roots^{(A),\vee}$, resp.~$\Roots_1^{(A),\vee}$, resp.~$\Roots_2^{(A),\vee}$, the image of $\Roots^{(A)}$, resp.~$\Roots^{(A)}_1$, resp.~$\Roots^{(A)}_2$, under the assignment $\gamma \mapsto \gamma^\vee$.

The following proposition is essentially proved in~\cite{haines,haines2}; see also~\cite{adler-lansky-2}. 

\begin{prop}
\phantomsection
\label{prop:root.syst.fix.pts}
\begin{enumerate}
\item
The quadruple $(M_A / (M_A)_{\mathrm{tor}},\Roots_1^{(A)}, (M^\vee)^A,\Roots_1^{(A),\vee})$ is a reduced root datum, with basis $\Delta^{(A)}$ and Weyl group $(W_\Phi)^A$.
\item
The quadruple $(M_A / (M_A)_{\mathrm{tor}},\Roots_2^{(A)}, (M^\vee)^A,\Roots_2^{(A),\vee})$ is a reduced root datum, with basis $\Delta_2^{(A)}$ and Weyl group $(W_\Phi)^A$.
\item
\label{it:root-data-3}
The quadruple $(M_A / (M_A)_{\mathrm{tor}},\Roots^{(A)}, (M^\vee)^A,\Roots^{(A),\vee})$ is a (not necessarily reduced) root datum, with basis $\Delta^{(A)}$ and Weyl group $(W_\Phi)^A$.
\end{enumerate}
\end{prop}

In the rest of this subsection, we briefly explain how this statement can be recovered from the analysis in~\S\ref{ss:fixed-pts-alg-closed-fields}. We first consider the first two parts. Assume that $S$ is the spectrum of an algebraically closed field $\bk$, and consider the pinned reductive group scheme $\bG$ over $S$ with root datum $\Phi$. By Proposition~\ref{prop:fix-pts-fields} and Remark~\ref{rmk:field-root-calc}, $(\bG^A)^\circ_{\mathrm{red}}$ is a connected reductive algebraic group, with maximal torus $(\bT^A)^\circ_{\mathrm{red}}$ (whose lattice of characters identifies with $M_A/(M_A)_{\mathrm{tor}}$), and its root system is $\Roots_1^{(A)}$ if $\mathrm{char}(\bk)\neq 2$, and $\Roots_2^{(A)}$ if $\mathrm{char}(\bk)= 2$. 
Remark~\ref{rmk:field-root-calc} also describes its coroots: they are the cocharacters constructed in Step~3 of the proof of Proposition~\ref{prop:fix-pts-fields}, in~\eqref{eqn:coroot-formula-pre}.  Note that that formula agrees with~\eqref{eqn:coroot-formula}.

In each case it is easily seen that $\Delta^{(A)}$ is a basis, and it follows from Proposition~\ref{prop:fix-pts-fields}\eqref{it:fields-Weyl-gp} that the corresponding Weyl group is $(W_\Phi)^A$.

This justifies the first two cases in Proposition~\ref{prop:root.syst.fix.pts}. 
The third case follows, using the standard observation that a union of root systems of types $\mathsf{B}_n$ and $\mathsf{C}_n$ produces a nonreduced root system of type $\mathsf{BC}_n$.

\begin{rmk}
In the setting of Corollary~\ref{cor:connected-fixed-pts-reductive}, from the pinning of $\bG$ one can obtain a pinning of $(\bG^A)^\circ$, with associated root datum $(M_A / (M_A)_{\mathrm{tor}},\Roots_1^{(A)}, (M^\vee)^A,\Roots_1^{(A),\vee})$.
\end{rmk}

\subsection{Weyl group}
\label{ss:Weyl-gp}

In this subsection we prove some results on the interplay between the maximal torus $\bT$, its normalizer $\bN$, the Weyl group $\bW$, and fixed points.

Let us consider once again a pinned reductive group scheme~\eqref{eqn:pinned-gp}
over $S$, and our group $A$ which acts on $\bG$ by pinned automorphisms. As in~\S\ref{ss:fixed-pts-alg-closed-fields} (but now over an arbitrary base), we denote by $\bN$ the normalizer of $\bT$ in $\bG$. By~\cite[Proposition~2.1.2]{conrad}, $\bN$ is a smooth subgroup scheme of $\bG$. Denote by $\lambda$ the sum of the positive coroots in $\Roots^\vee$; then $\lambda$ defines a cocharacter $\bG_{\mathrm{m},S} \to \bT$, which is $A$-invariant and hence takes values in $\bT^A$. As explained e.g.~in~\cite[Theorem~5.1.13]{conrad}, $\bT$ coincides with the centralizer of $\lambda$; in particular, $\bT$ is its own centralizer in $\bG$.
The Weyl group
$\bW$ is the quotient sheaf $\bN/\bT$ for the fppf topology. By~\cite[Proposition~5.1.6]{conrad}, $\bW$ is representable by the constant group scheme $(W_\Phi)_S$ over $S$, where $W_\Phi$ is the Weyl group of the root datum $\Phi=(M,\Roots, M^\vee, \Roots^\vee)$.

\begin{lem}
\label{lem:centralizer.norm.torus.weyl.gp}
	The following properties hold:
\begin{enumerate}
\item 
\label{it:Weyl-group-1}
the centralizer of $\bT^A$ in $\bG$ is representable by $\bT$; in particular, the centralizer of $\bT^A$ in $\bG^A$ is representable by the closed subgroup scheme $\bT^A$;
\item
\label{it:Weyl-group-2}
the normalizer of $\bT^A$ in $\bG$ is contained in $\bN$; in particular,
the normalizer of $\bT^A$ in $\bG^A$ is representable by the closed subgroup scheme $\bN^A$;
\item
\label{it:Weyl-group-3}
the quotient sheaf $\bN^A/\bT^A$ is representable by the constant $S$-group scheme $((W_\Phi)^A)_S$.
	\end{enumerate}
\end{lem}

\begin{proof}
\eqref{it:Weyl-group-1}
As explained above, the cocharacter $\lambda$ takes values in $\bT^A$; its centralizer in $\bG$ (i.e.~$\bT$) therefore contains the centralizer of $\bT^A$. We deduce that $\mathrm{Z}_{\bG}(\bT^A)=\bT$, and then that $\mathrm{Z}_{\bG^A}(\bT^A)=\bT^A$.

\eqref{it:Weyl-group-2} 
The normalizer $\mathrm{N}_{\bG}(\bT^A)$ of $\bT^A$ in $\bG$ must preserve the centralizer $\mathrm{Z}_{\bG}(\bT^A)=\bT$.
We deduce an inclusion $\mathrm{N}_{\bG}(\bT^A) \subset \bN$. Intersecting with $\bG^A$, it follows that $\mathrm{N}_{\bG^A}(\bT^A) \subset \bN^A$. The opposite inclusion is clear.

\eqref{it:Weyl-group-3} 
The quotient sheaf $\bN^A/\bT^A$ injects into $\bW^A = ((W_\Phi)^A)_S$ by definition. 
The proof of surjectivity is similarly to the proof of surjectivity of~\eqref{eqn:NA-WA}: either by~\cite[Corollaire~3.5]{hee} or because $\bW^A$ is the Weyl group of the root datum $(M_A / (M_A)_{\mathrm{tor}},\Roots^{(A)}, (M^\vee)^A,\Roots^{(A),\vee})$ (see Lemma~\ref{prop:root.syst.fix.pts}), this group is generated by the simple reflections associated with equivalence classes in $\Roots_+$ which intersect $\Delta$. If $E$ is such an equivalence class of type (i$'$), resp.~(ii$'$), then 
the corresponding reflection is the image of 
$\imath_E \left( \begin{smallmatrix} 0 & 1 \\ -1 & 0 \end{smallmatrix} \right)$, resp.~$\imath_E \left( \begin{smallmatrix} 0 & 0 & 1 \\ 0 & -1 & 0 \\ 1 & 0 & 0 \end{smallmatrix} \right)$.
\end{proof}
	 
\subsection{Parabolic and Levi subgroups}
\label{ss:parabolic-Levi}

Let us consider once again a pinned reductive group scheme~\eqref{eqn:pinned-gp}
over $S$, and our group $A$ which acts on $\bG$ by pinned automorphisms.
Recall that given an $S$-scheme $X$ which is affine over $S$ and endowed with an action of a diagonalizable group scheme $\Diag_S(N)$, following~\cite{mayeux}, for any submonoid $Q \subset N$ we have a closed subscheme $X^Q \subset X$ called the attractor scheme associated with $Q$. (This construction generalizes the classical notion of attractor for an action of $\bG_{\mathrm{m}}$ used e.g.~in~\cite{CGP10,conrad}.)
In this subsection we will consider this construction in the case of the action of $\bT=\Diag_S(M)$, resp.~$\bT^A = \Diag_S(M_A)$, on $X=\bG$, resp.~$X=\bG^A$.

Recall the root datum 
\[
(M_A / (M_A)_{\mathrm{tor}},\Roots^{(A)}, (M^\vee)^A,\Roots^{(A),\vee})
\]
considered in Proposition~\ref{prop:root.syst.fix.pts}\eqref{it:root-data-3}, and its basis $\Delta^{(A)}$. We have a canonical identification $\Delta/A \simto \Delta^{(A)}$, and therefore a canonical bijection between subsets of $\Delta^{(A)}$ and $A$-stable subsets of $\Delta$. Consider an $A$-stable subset $\Gamma \subset \Delta$, and denote by $\Gamma^{(A)} \subset \Delta^{(A)}$ its image in $\Delta^{(A)}$. To $\Gamma$ we can associate a parabolic subgroup $\bP$ of $\bG$, whose Lie algebra is the sum of $\Lie(\bB)$ and the root subspaces in $\Lie(\bG)$ associated with the roots which belong to $\Z\Gamma = \sum_{\gamma \in \Gamma} \Z \cdot \gamma$. Following~\cite[Example~5.2.2]{conrad}, this subgroup scheme can be realized as the attractor subscheme associated with the conjugation action of $\bG_{\mathrm{m},S}$ on $\bG$ via a cocharacter $\mu \in M^\vee$ which satisfies
\[
\begin{cases}
\langle \mu, \alpha^\vee \rangle > 0 & \text{if $\alpha \in \Delta \smallsetminus \Gamma$;} \\
\langle \mu, \alpha^\vee \rangle = 0 & \text{if $\alpha \in \Gamma$.}
\end{cases}
\]
We also have the Levi factor $\bM$ of $\bP$ containing $\bT$, which can be realized as the fixed point subscheme associated with the same action of $\bG_{\mathrm{m},S}$, see~\cite[Proposition~5.4.5 and its proof]{conrad}, and its unipotent radical $\bU_{\bP}$, see~\cite[Corollary~5.2.5]{conrad}.
The theory of~\cite{mayeux} allows us to make this construction more canonical, by suppressing the choice of the cocharacter $\mu$. Namely,
denote by $N_\Gamma$ the submonoid of $M$ generated by $\Delta \sqcup (-\Gamma)$. The following claim is an immediate generalization of~\cite[Proposition~7.3]{mayeux}. 

\begin{lem}
\label{lem:parabolic-Levi-attractor}
We have
\[
\bP = \bG^{N_\Gamma}, \quad \bM = \bG^{\Z\Gamma}.
\]
\end{lem}

\begin{proof}
The proof is the same as that of~\cite[Proposition~7.3]{mayeux}.
Namely, choose $\mu \in M^\vee$ as above, which we now consider as a morphism $M \to \Z$. The induced morphism $\bG_{\mathrm{m},S} = \Diag_S(\Z) \to \Diag_S(M)=\bT$ is precisely $\mu$ seen as a cocharacter. By~\cite[Lemma~5.6]{mayeux} we have $\bG^{\mu^{-1}(\Z_{\geq 0})} = \bG^{\Z_{\geq 0}}$ where on the left-hand side we regard $\bG$ with the action of $\bT$, and on the right-hand side we regard $\bG$ with the action of $\bG_{\mathrm{m},S}$ via $\mu$. As explained above the right-hand side is known to coincide with $\bP$. Using the fact that $N_\Gamma \subset \mu^{-1}(\Z_{\geq 0})$, in view of~\cite[Remark~3.5]{mayeux} we deduce that $\bG^{N_\Gamma} \subset \bP$. On the other hand we have $\bP = \bM \ltimes \bU_\bP$. It is easily seen that both $\bM$ and $\bU_\bP$ are contained in $\bG^{N_\Gamma}$, which implies that $\bP \subset \bG^{N_\Gamma}$ and finishes the proof of the first claim. The proof of the second one is similar.
\end{proof}

Now, let us consider the submonoid $N_\Gamma^{(A)}$ of $M_A$ generated by $\Delta^{(A)} \sqcup (-\Gamma^{(A)})$ and the subgroup $\Z\Gamma^{(A)}$ generated by $\Gamma^{(A)}$.

\begin{prop}
	The natural morphisms 
	\[
	\bP^A \to (\bG^A)^{N_\Gamma^{(A)}} \quad \text{ and } \quad  \bM^A \to (\bG^A)^{\Z\Gamma^{(A)}}
	\]
	are isomorphisms. 
\end{prop}
 
 \begin{proof}
 The proof is similar to that of Lemma~\ref{lem:parabolic-Levi-attractor}; we explain the first case, and leave the easy modifications to treat the second case to the reader. The cocharacter $\mu$ considered above can be chosen to be $A$-invariant. 
(In fact one can reduce the situation to that of a semisimple and simply connected group, and then average any cocharacter satisfying our conditions along the action of $A$.)
 Then $\mu$ takes values in $\bT^A$, or in other words the associated morphism $M \to \Z$ (still denoted $\mu$) factors through a morphism $\overline{\mu} : M_A \to \Z$. Since the action of $\bG_{\mathrm{m},S}$ via $\mu$ commutes with the action of $A$ we have
 \[
 \bP^A = (\bG^{\Z_{\geq 0}})^A = (\bG^A)^{\Z_{\geq 0}}.
 \]
 By~\cite[Lemma~5.6]{mayeux} we have $(\bG^A)^{\Z_{\geq 0}} = (\bG^A)^{\overline{\mu}^{-1}(\Z_{\geq 0})}$. Since $N_\Gamma^{(A)} \subset \overline{\mu}^{-1}(\Z_{\geq 0})$, we deduce that
 \[
 (\bG^A)^{N_\Gamma^{(A)}} \subset \bP^A.
 \]
 On the other hand, if $T$ is an $S$-scheme and $x : T \to \bP^A$ is a $T$-point of $\bP^A$, then since $\bP=\bG^{N_\Gamma}$ we have a canonical morphism
 \[
 x' : \mathrm{A}_S(N_\Gamma) \to \bG
 \]
 where the left-hand side is the $S$-scheme associated with the monoid $N_\Gamma$ following~\cite[\S 3]{mayeux}. This morphism is easily seen to be $A$-invariant; it therefore restricts to a morphism $x'' : (\mathrm{A}_S(N_\Gamma))^A \to \bG^A$. We also have a canonical morphism $N_\Gamma \to N_\Gamma^{(A)}$, which induces a morphism $\mathrm{A}_S(N_\Gamma^{(A)}) \to \mathrm{A}_S(N_\Gamma)$. This morphism is $A$-equivariant with respect to the trivial action on its domain, so it factors through a morphism $\mathrm{A}_S(N_\Gamma^{(A)}) \to (\mathrm{A}_S(N_\Gamma))^A$. Composing with $x''$ we obtain a morphism
 \[
 \mathrm{A}_S(N_\Gamma^{(A)}) \to \bG^A.
 \]
 This construction shows that $\bP^A \subset (\bG^A)^{N_\Gamma^{(A)}}$, which finishes the proof.  
 \end{proof}
  
 In particular, given $\gamma \in \Delta^{(A)}$, we can consider the parabolic subgroup and its Levi factor $\bM_\gamma \subset \bP_\gamma \subset \bG$ associated with the $A$-orbit in $\Delta$ corresponding to $\gamma$, and the corresponding fixed points $(\bM_\gamma)^A \subset \bG$. The following claim will be required in~\cite{alrr}.
 
 \begin{cor}
 \label{cor.generation.levis}
 The unique subgroup scheme of $\bG^A$ that contains $(\bM_\gamma)^A$ for every $\gamma \in \Delta^{(A)}$ is $\bG^A$ itself.
 \end{cor}

\begin{proof}
Let $\bH \subset \bG^A$ be a subgroup scheme that contains all the subgroups $(\bM_\gamma)^A$ for $\gamma \in \Delta^{(A)}$. Then $\bH$ contains $\bT^A$, and also the elements in $\bN^A$ corresponding to the simple reflections in the Weyl group $\bW^A$ constructed in the proof of Lemma~\ref{lem:centralizer.norm.torus.weyl.gp}\eqref{it:Weyl-group-3}. Since these elements generate $\bW^A$, we deduce that $\bH$ contains $\bN^A$. Since any root in $\Roots^{(A)}_1$ is $\bW^A$-conjugate to a root in $\Delta^{(A)}$ (see~\S\ref{ss:root-data}), we then deduce that $\bH$ contains each subgroup $(\bU_E)^A$ with $E$ an equivalence class in $\Roots_+$, and therefore that it contains $\bC^A$ (see~\eqref{eqn:decomp-CA} and~\eqref{eqn:ua-product}). Finally,
since the multiplication morphism $\bC^A \times \bC^A \to \bG^A$ is an epimorphism (see Remark~\ref{rmk:f-epi}), this implies that $\bH=\bG^A$.
\end{proof}

\subsection{Center and isogenies}

Let us consider once again a pinned reductive group scheme~\eqref{eqn:pinned-gp}
over $S$, and our group $A$ which acts on $\bG$ by pinned automorphisms.
Recall the definition of the center of a group scheme; see~\cite[Definition~2.2.1]{conrad}. (By definition, this center is a sheaf of groups, which is not necessarily representable.)
In particular let $\bZ \subset \bG$ denote the center of $\bG$.
By~\cite[Theorem~3.3.4]{conrad}, $\bZ$ is representable by a diagonalizable group scheme; more explicitly we have
	\[ 
	\bZ= \bigcap_{\alpha \in \Delta} \ker(\alpha) = \Diag_S(M/\Z\Roots).
	\]
By functoriality, $\bZ$ is preserved by the $A$-action, so we may form its fixed-point scheme $\bZ^A$, which is again a diagonalizable group scheme: more explicitly, by Lemma~\ref{lem:torus_fix_pts_diag} we have
\begin{equation}
\label{eqn:ZA-Diag}
\bZ^A = \Diag_S((M/\Z\Roots)_A).
\end{equation}

\begin{lem}
	The closed subgroup scheme $\bZ^A$ represents the center of $\bG^A$.
\end{lem}

\begin{proof}
Denote by $\bZ'$ the center of $\bG^A$.
	Since $\bZ$ is the center of $\bG$, we have $\bZ^A \subset \bZ'$. On the other hand, we know by Lemma~\ref{lem:centralizer.norm.torus.weyl.gp} that $\bT^A$ is its own centralizer, so $\bZ'$ is contained in $\bT^A$. Considering the $\bT^A$-action on each $(\bU_E)^A$,
	we obtain that
	\[ 
	\bZ' \subset \bigcap_{\alpha \in \Delta^{(A)}} \ker(\alpha),
	\]
	which in view of~\eqref{eqn:ZA-Diag} shows that $\bZ' \subset \bZ^A$, and hence that $\bZ'=\bZ^A$.
\end{proof}

Recall the group scheme $\bG_\sconn$ considered in~\S\ref{ss:reduction-scqs}, and denote by $\bZ_\sconn$ its center. The natural morphism $f : \bG_\sconn \to \bG$ restricts to a morphism $\bZ_\sconn \to \bZ$. It is a standard fact that the natural morphism
\[
\bG_{\sconn} \times_S^{\bZ_\sconn} \bZ \to \bG
\]
is an isomorphism, where the left-hand side is the (fppf) quotient of $\bG_{\sconn} \times_S \bZ$ by the action of $\bZ_{\sconn}$ defined by $z \cdot (g,h) = (gz^{-1}, f(z)h)$. In other words, we have an exact sequence of fppf sheaves of groups
\begin{equation}
\label{eqn:es-Gsc}
1 \to \bZ_{\sconn} \to \bG_{\sconn} \times_S \bZ \to \bG \to 1  
\end{equation}
where the first morphism is the natural antidiagonal embedding. The next lemma provides a version of this result for the group $\bG^A$, which will be used in~\cite{alrr}.

\begin{prop}
\label{prop:fib.prod.sc.cent}
	The natural map
\[
(\bG_{\sconn})^A \times_S^{(\bZ_{\sconn})^A} \bZ^A \to \bG^A
\]
is an isomorphism, where the left-hand side is the quotient of $(\bG_{\sconn})^A \times_S \bZ^A$ by the action of $(\bZ_{\sconn})^A$ induced by the action of $\bZ_{\sconn}$ on $\bG_{\sconn} \times_S \bZ$ considered above.
\end{prop}

\begin{proof}
The given map is clearly a monomorphism of group objects in fppf sheaves, as taking $A$-fixed points is a left exact functor. To conclude the proof, it therefore suffices to check surjectivity of the morphism $\bG_{\sconn}^A\times_S \bZ^A\to \bG^A$ in the fppf topology. Since the multiplication map $\bC^A\times_S \bC^A\to \bG^A$ is an epimorphism of fppf sheaves (see Remark~\ref{rmk:f-epi}), we are reduced to showing surjectivity of $\bC_\sconn^A\times_S \bZ^A\to \bC^A $ where $\bC_\sconn$ is the big cell of $\bG_\sconn$. But passing to simply connected covers does not affect root groups so, due to the decomposition \eqref{eqn:decomp-CA}, it is enough to show the surjectivity of the morphism
\[
\bT_\sconn^A\times_S \bZ^A \to \bT^A.
\]
In order to see this, we remark that $\bT/\bZ=\bT_{\mathrm{adj}} := \Diag_S(\Z\Roots)$. By left exactness of fixed points, we can embed $\bT^A/\bZ^A$ into $(\bT_{\mathrm{adj}})^A$, reducing the problem to proving that the morphism 
	$(\bT_{\sconn})^A \to (\bT_{\mathrm{adj}})^A$ is surjective. In fact, in view of~\cite[Exp.~VIII, \S 3]{sga3.2} or~\cite[\S 5.3]{oesterle}, to prove this claim it suffices to prove that the natural morphism $(\Z \Roots)_A \to (M_\sconn)_A$ is injective, which follows from the fact that these $\Z$-modules are free of finite rank (because $\Z \Roots$ and $M_\sconn$ both have a basis permuted by $A$), and that the given morphism becomes an isomorphism after tensor product with $\Q$. 
\end{proof}

 \subsection{Further study of the case of \texorpdfstring{$\mathrm{SL}_3$}{SL3}}
 \label{ss:action-SL3}
 
In this subsection we prove a technical statement regarding the group scheme $(\mathrm{SL}_{3,\Z_2})^{\Z/2\Z}$ 
for the action considered in~\S\ref{sss:sl3_action} (see Proposition~\ref{prop:flat-map-non-reduced-sch-dominant}) that will be required in the companion paper~\cite{alrr}.

If we set
\[
n := \begin{pmatrix}
0 & 0 & 1 \\
0 & -1 & 0 \\
1 & 0 & 0
\end{pmatrix},
\]
then $n$ is a $\Z$-point of $(\mathrm{SL}_{3,\Z})^{\Z/2\Z}$. Let us denote by $\mathrm{C}_{3,\Z}$ the big cell constructed from the pinning of Example~\ref{ex:roots-A2n}. We also let $\mathrm{U}_{3,\Z}$ be as in~\S\ref{sss:sl3_action}, $\mathrm{U}^-_{3,\Z}$ be the similar group of lower triangular matrices, and $\mathrm{T}_{3,\Z}$ be the maximal torus of Example~\ref{ex:roots-A2n}. With this notation, $(\mathrm{SL}_{3,\Z})^{\Z/2\Z}$
has an affine open covering with two open subsets given explicitly by
\[
(\mathrm{C}_{3,\Z})^{\Z/2\Z} \cong (\mathrm{U}_{3,\Z}^-)^{\Z/2\Z} \times_{\Spec(\Z)} (\mathrm{T}_{3,\Z})^{\Z/2\Z} \times_{\Spec(\Z)} (\mathrm{U}_{3,\Z})^{\Z/2\Z}
\]
and
\[
n \cdot (\mathrm{C}_{3,\Z})^{\Z/2\Z} \cong (\mathrm{U}_{3,\Z}^-)^{\Z/2\Z} \times_{\Spec(\Z)} (\mathrm{T}_{3,\Z})^{\Z/2\Z} \times_{\Spec(\Z)} (\mathrm{U}_{3,\Z})^{\Z/2\Z}.
\]
(To see that these two subschemes cover $(\mathrm{SL}_{3,\Z})^{\Z/2\Z}$, it suffices to check that they contain all points over $\Spec(\Z[\frac{1}{2}])$ and over $\Spec(\F_2)$, and this follows from the analysis in~\S\ref{ss:fixed-pts-alg-closed-fields}.) As explained in~\S\ref{sss:sl3_action}, we have
\[
(\mathrm{U}_{3,\Z}^-)^{\Z/2\Z} \cong (\mathrm{U}_{3,\Z})^{\Z/2\Z} \cong \Spec(\Z[x,y] / (x^2-2y)),
\]
and the considerations in~\S\ref{ss:ex-diag-gps} show that we have an isomorphism
\begin{equation}
\label{eqn:fixed-pts-torus-SL3}
\mathbb{G}_{\mathrm{m},\Z} \simto (\mathrm{T}_{3,\Z})^{\Z/2\Z}
\end{equation}
given explicitly by
\[
a \mapsto \mathrm{diag}(a,1,a^{-1}).
\]

For any commutative ring $A$, we denote by $\mathrm{SL}_{3,A}$, $\mathrm{C}_{3,A}$, etc., the schemes obtained by base change to $\Spec(A)$. We will be particularly interested in the cases where $A$ is $\Z_2$ or $\F_2$. Our goal is to prove the following statement.

\begin{prop}
\label{prop:flat-map-non-reduced-sch-dominant}
Let $\bH$ be a flat affine group scheme over $\Z_2$, and let
\[
\pi : \bH \to (\mathrm{SL}_{3,\Z_2})^{\Z/2\Z}
\]
be a morphism of group schemes such that
$\pi_{|\Spec(\F_2)}$ is surjective (at the level of topological spaces).
Then the schematic image of $\pi_{|\Spec(\F_2)}$ is $(\mathrm{SL}_{3,\F_2})^{\Z/2\Z}$.
\end{prop}

We start with two preliminary lemmas.

\begin{lem}
\label{lem:subgroups-SL3}
The only closed subgroup schemes of $(\mathrm{SL}_{3,\F_2})^{\Z/2\Z}$ containing the subgroup $(\mathrm{SL}_{3,\F_2})_{\mathrm{red}}^{\Z/2\Z}$ are $(\mathrm{SL}_{3,\F_2})_{\mathrm{red}}^{\Z/2\Z}$ and $(\mathrm{SL}_{3,\F_2})^{\Z/2\Z}$.
\end{lem}

\begin{proof}
Let $\mathbf{K}$ be a closed subgroup scheme of $(\mathrm{SL}_{3,\F_2})^{\Z/2\Z}$ containing $(\mathrm{SL}_{3,\F_2})_{\mathrm{red}}^{\Z/2\Z}$. The closed immersion $(\mathrm{T}_{3,\F_2})^{\Z/2\Z} \hookrightarrow (\mathrm{SL}_{3,\F_2})^{\Z/2\Z}$ factors through $(\mathrm{SL}_{3,\F_2})_{\mathrm{red}}^{\Z/2\Z}$, and hence through $\mathbf{K}$, and $(\mathrm{T}_{3,\F_2})^{\Z/2\Z}$ identifies with $\mathbb{G}_{\mathrm{m}, \F_2}$, see~\eqref{eqn:fixed-pts-torus-SL3}. We will consider the attractor, resp.~repeller, scheme associated with the conjugation action of this subgroup on $\mathrm{SL}_{3,\F_2}$ and $(\mathrm{SL}_{3,\F_2})_{\mathrm{red}}^{\Z/2\Z}$, in the sense of~\cite[\S 2.1]{CGP10}. The attractor, resp.~repeller, for the action on $\mathrm{SL}_{3,\F_2}$ is the standard positive, resp.~negative, Borel subgroup, see e.g.~\cite[Theorem~5.1.13 and its proof]{conrad}. Hence the corresponding attractor, resp.~repeller, for the action on $(\mathrm{SL}_{3,\F_2})^{\Z/2\Z}$ identifies with
\[
(\mathrm{T}_{3,\F_2})^{\Z/2\Z} \times_{\Spec(\F_2)} (\mathrm{U}_{3,\F_2})^{\Z/2\Z}, \quad \text{resp.} \quad (\mathrm{U}^-_{3,\F_2})^{\Z/2\Z} \times_{\Spec(\F_2)} (\mathrm{T}_{3,\F_2})^{\Z/2\Z}.
\]
In view of~\cite[Proposition~2.1.8(3)]{CGP10}, we deduce that
\begin{multline*}
\mathbf{K} \cap (\mathrm{C}_{3,\F_2})^{\Z/2\Z} = \\
\bigl( \mathbf{K} \cap (\mathrm{U}^-_{3,\F_2})^{\Z/2\Z} \bigr) \times_{\Spec(\F_2)} (\mathrm{T}_{3,\F_2})^{\Z/2\Z} \times_{\Spec(\F_2)}  \bigl( \mathbf{K} \cap (\mathrm{U}_{3,\F_2})^{\Z/2\Z} \bigr).
\end{multline*}
Now, we observe that the matrix description for $(\mathrm{U}_{3,\F_2})^{\Z/2\Z}$ provides a short exact sequence of group schemes
\[
1 \to (\mathrm{U}_{3,\F_2})^{\Z/2\Z}_{\mathrm{red}} \to (\mathrm{U}_{3,\F_2})^{\Z/2\Z} \to \alpha_{2,\F_2} \to 1.
\]
Since $\mathbf{K}$ contains $(\mathrm{U}_{3,\F_2})^{\Z/2\Z}_{\mathrm{red}}$, it follows that $\mathbf{K} \cap (\mathrm{U}_{3,\F_2})^{\Z/2\Z}$ is either $(\mathrm{U}_{3,\F_2})^{\Z/2\Z}$ or $(\mathrm{U}_{3,\F_2})^{\Z/2\Z}_{\mathrm{red}}$. 
Since $\mathbf{K}(\F_2)$ contains the image of the element $n$ considered above, we have
\[
\mathbf{K} \cap (\mathrm{U}^-_{3,\F_2})^{\Z/2\Z} = 
\begin{cases}
(\mathrm{U}^-_{3,\F_2})^{\Z/2\Z} & \text{if $\mathbf{K} \cap (\mathrm{U}_{3,\F_2})^{\Z/2\Z} = (\mathrm{U}_{3,\F_2})^{\Z/2\Z}$;} \\
(\mathrm{U}^-_{3,\F_2})^{\Z/2\Z}_{\mathrm{red}} & \text{if $\mathbf{K} \cap (\mathrm{U}_{3,\F_2})^{\Z/2\Z} = (\mathrm{U}_{3,\F_2})^{\Z/2\Z}_{\mathrm{red}}$.}
\end{cases}
\]
As a conclusion, we have either
\[
\mathbf{K} \cap (\mathrm{C}_{3,\F_2})^{\Z/2\Z} = (\mathrm{C}_{3,\F_2})^{\Z/2\Z} \quad
\text{or} \quad
\mathbf{K} \cap (\mathrm{C}_{3,\F_2})^{\Z/2\Z} = (\mathrm{C}_{3,\F_2})_{\mathrm{red}}^{\Z/2\Z}.
\]
Using once again the fact that $\mathbf{K}(\F_2)$ contains the image of $n$, and the description of the open cover of  $(\mathrm{SL}_{3,\Z})^{\Z/2\Z}$ considered above,
in the first case we deduce that $\mathbf{K}=(\mathrm{SL}_{3,\F_2})^{\Z/2\Z}$, and in the second case that $\mathbf{K}=(\mathrm{SL}_{3,\F_2})^{\Z/2\Z}_{\mathrm{red}}$.
\end{proof}
 
The second lemma uses the notion of dilatation (or affine blow-up) from~\cite{mrr}. We will apply this construction to the scheme $(\mathrm{SL}_{3,\Z_2})^{\Z/2\Z}$, the principal subscheme $(\mathrm{SL}_{3,\F_2})^{\Z/2\Z}$, and either the closed subscheme $(\mathrm{SL}_{3,\F_2})_{\mathrm{red}}^{\Z/2\Z} \subset (\mathrm{SL}_{3,\F_2})^{\Z/2\Z}$ or the closed subscheme $(\mathrm{T}_{3,\F_2})^{\Z/2\Z} \subset (\mathrm{SL}_{3,\F_2})^{\Z/2\Z}$. The construction of~\cite{mrr} provides two affine schemes
\[
\mathrm{Bl}^{(\mathrm{SL}_{3,\F_2})^{\Z/2\Z}}_{(\mathrm{SL}_{3,\F_2})_{\mathrm{red}}^{\Z/2\Z}}\Bigl( (\mathrm{SL}_{3,\Z_2})^{\Z/2\Z} \Bigr) \quad \text{and} \quad
\mathrm{Bl}^{(\mathrm{SL}_{3,\F_2})^{\Z/2\Z}}_{(\mathrm{T}_{3,\F_2})^{\Z/2\Z}} \Bigl( (\mathrm{SL}_{3,\Z_2})^{\Z/2\Z} \Bigr)
\]
endowed with canonical morphisms to $(\mathrm{SL}_{3,\Z_2})^{\Z/2\Z}$. Moreover, by the universal property of dilatations (see~\cite[Proposition~2.6]{mrr}) there exists a canonical morphism
\begin{equation}
\label{eqn:morph-blowups}
\mathrm{Bl}^{(\mathrm{SL}_{3,\F_2})^{\Z/2\Z}}_{(\mathrm{T}_{3,\F_2})^{\Z/2\Z}} \Bigl( (\mathrm{SL}_{3,\Z_2})^{\Z/2\Z} \Bigr) \to \mathrm{Bl}^{(\mathrm{SL}_{3,\F_2})^{\Z/2\Z}}_{(\mathrm{SL}_{3,\F_2})_{\mathrm{red}}^{\Z/2\Z}}\Bigl( (\mathrm{SL}_{3,\Z_2})^{\Z/2\Z} \Bigr)
\end{equation}
over $(\mathrm{SL}_{3,\Z_2})^{\Z/2\Z}$.

\begin{lem}
\label{lem:blowup-SL3}
The morphism~\eqref{eqn:morph-blowups} restricts to an isomorphism over $(\mathrm{C}_{3,\Z_2})^{\Z/2\Z}$.
\end{lem}

\begin{proof}
By compatibility of dilatations with flat base change (see~\cite[Lem\-ma~2.7]{mrr}), we have
\[
\mathrm{Bl}^{(\mathrm{SL}_{3,\F_2})^{\Z/2\Z}}_{(\mathrm{SL}_{3,\F_2})_{\mathrm{red}}^{\Z/2\Z}}\Bigl( (\mathrm{SL}_{3,\Z_2})^{\Z/2\Z} \Bigr) \times_{(\mathrm{SL}_{3,\Z_2})^{\Z/2\Z}} (\mathrm{C}_{3,\Z_2})^{\Z/2\Z} \cong \mathrm{Bl}^{(\mathrm{C}_{3,\F_2})^{\Z/2\Z}}_{(\mathrm{C}_{3,\F_2})_{\mathrm{red}}^{\Z/2\Z}}\Bigl( (\mathrm{C}_{3,\Z_2})^{\Z/2\Z} \Bigr),
\]
with
\begin{multline*}
\mathrm{Bl}^{(\mathrm{C}_{3,\F_2})^{\Z/2\Z}}_{(\mathrm{C}_{3,\F_2})_{\mathrm{red}}^{\Z/2\Z}}\Bigl( (\mathrm{C}_{3,\Z_2})^{\Z/2\Z} \Bigr) = \\
\mathrm{Bl}^{\Spec(\F_2[x,y,z^{\pm 1}, x',y'] / (x^2, (x')^2))}_{\Spec(\F_2[y,z^{\pm 1},y'])}\Bigl( \Spec(\Z_2[x,y,z^{\pm 1}, x',y'] / (x^2-2y, (x')^2-2y'))\Bigr).
\end{multline*}
Now one checks that the right-hand side identifies with
\[
\Spec(\Z_2[u,y,z^{\pm 1}, u',y'] / (2u^2-y, 2(u')^2-y')),
\]
with the morphism to $\Spec(\Z_2[x,y,z^{\pm 1}, x',y'] / (x^2-2y, (x')^2-2y'))$ corresponding to the ring homomorphism
\[
\Z_2[x,y,z^{\pm 1}, x',y'] / (x^2-2y, (x')^2-2y') \to \Z_2[u,y,z^{\pm 1}, u',y'] / (2u^2-y, 2(u')^2-y')
\]
defined by $x \mapsto 2u$ and $x' \mapsto 2u'$. (In fact the given scheme satisfies the universal property of~\cite[Proposition~2.6]{mrr}.) This description shows that the restriction of the canonical morphism
\[
\mathrm{Bl}^{(\mathrm{C}_{3,\F_2})^{\Z/2\Z}}_{(\mathrm{C}_{3,\F_2})_{\mathrm{red}}^{\Z/2\Z}}\Bigl( (\mathrm{C}_{3,\Z_2})^{\Z/2\Z} \Bigr) \to (\mathrm{C}_{3,\Z_2})^{\Z/2\Z}
\]
to $\Spec(\F_2)$ factors through $\Spec(\F_2[z^{\pm 1}])$ since $x,x',y,y'$ get sent to $0$; by the universal property of dilatations (and again compatibility of dilatations with flat base change) we deduce a canonical morphism
\begin{multline*}
 \mathrm{Bl}^{(\mathrm{SL}_{3,\F_2})^{\Z/2\Z}}_{(\mathrm{SL}_{3,\F_2})_{\mathrm{red}}^{\Z/2\Z}}\Bigl( (\mathrm{SL}_{3,\Z_2})^{\Z/2\Z} \Bigr) \times_{(\mathrm{SL}_{3,\Z_2})^{\Z/2\Z}} (\mathrm{C}_{3,\Z_2})^{\Z/2\Z} \to
\\
\mathrm{Bl}^{(\mathrm{SL}_{3,\F_2})^{\Z/2\Z}}_{(\mathrm{T}_{3,\F_2})^{\Z/2\Z}} \Bigl( (\mathrm{SL}_{3,\Z_2})^{\Z/2\Z} \Bigr) \times_{(\mathrm{SL}_{3,\Z_2})^{\Z/2\Z}} (\mathrm{C}_{3,\Z_2})^{\Z/2\Z}.
\end{multline*}
We claim that this morphism is an inverse to the restriction of~\eqref{eqn:morph-blowups}. Indeed, both are flat affine schemes over $\Spec(\Z_2)$ with generic fiber equal to $(\mathrm{SL}_{3,\Q_2})^{\Z/2\Z}$, so to prove the claim it is enough to see that their global sections inside those of $(\mathrm{SL}_{3,\Q_2})^{\Z/2\Z}$ coincide. But each containment is implied by the existence of a map of spectra in the opposite direction.
\end{proof}

\begin{proof}[Proof of Proposition~\ref{prop:flat-map-non-reduced-sch-dominant}]
Let $\pi$ be as in the statement, and assume that the sche\-matic image of $\pi_{|\Spec(\F_2)}$ is not $(\mathrm{SL}_{3,\F_2})^{\Z/2\Z}$. By Lemma~\ref{lem:subgroups-SL3}, we conclude that the schematic image of this morphism is $(\mathrm{SL}_{3,\F_2})_{\mathrm{red}}^{\Z/2\Z}$. 
By the universal property of dilatations (see~\cite[Proposition~2.6]{mrr}), we then have a unique factorization
\[
\pi : \bH \to \mathrm{Bl}^{(\mathrm{SL}_{3,\F_2})^{\Z/2\Z}}_{(\mathrm{SL}_{3,\F_2})_{\mathrm{red}}^{\Z/2\Z}}\Bigl( (\mathrm{SL}_{3,\Z_2})^{\Z/2\Z} \Bigr) \to  (\mathrm{SL}_{3,\Z_2})^{\Z/2\Z}.
\] 
Using Lemma~\ref{lem:blowup-SL3}, we deduce that the restriction of $\pi_{|\Spec(\F_2)}$ to $(\mathrm{C}_{3,\F_2})^{\Z/2\Z}$
factors through $(\mathrm{T}_{3,\F_2})^{\Z/2\Z}$, yielding a contradiction.
\end{proof}

	\bibliography{biblio.bib}
	\bibliographystyle{alpha}

\end{document}